\documentclass[11pt]{amsart}
\usepackage[utf8]{inputenc}
\usepackage{amsthm}
\usepackage{amsmath}
\usepackage{amssymb}
\usepackage{tikz}
\usetikzlibrary{patterns}
\usepackage{geometry}
\usepackage{stmaryrd}
\usepackage{dirtytalk}
\usepackage{amsaddr}
\usepackage{hyperref}
\usepackage{ulem}

\newtheorem{thm}{Theorem}[section]
\newtheorem{dfn}{Definition}[section]
\newtheorem{prop}{Proposition}[section]
\newtheorem{cor}{Corollary}[section]

\newtheorem{theorem}{Theorem}

\newcommand\M{\mathcal{M}}

\newcommand\teich{\tilde{\mathcal{H}}(\kappa)}

\newcommand\moduli{\mathcal{H}(\kappa)}
\newcommand\GL{GL_2^{+}(\mathbb{R})}

\author{Florent Ygouf}
\address{School of Mathematical Sciences, Tel Aviv University}
\email{florentygouf@mail.tau.ac.il}

\title{Nonarithmetic affine invariant orbifolds in $\mathcal{H}^{\MakeLowercase{odd}}(2,2)$ and $\mathcal{H}(3,1)$}

\begin{document}

\parindent=0em

\maketitle

\begin{abstract}
We classify the nonarithmetic affine invariant orbifolds of complex dimension 3 in $\mathcal{H}^{odd}(2,2)$ and $\mathcal{H}(3,1)$. 
\end{abstract}

\section{Introduction}

\subsection{Context and results}

Let $g\geq 1, n \geq 1$ and let $\kappa=(k_1,\cdots,k_n)$ be an integer partition of $2g-2$. We denote by $\moduli$ the stratum of the moduli space of translation surfaces of type $\kappa$, \textit{i.e} the space of isomorphism classes of triples $q = (X_q,\omega_q,z_q)$ where $X_q$ is a genus $g$ Riemann surface, $\omega_q$ is non-zero holomorphic $1$-form on $X_q$ and $z_q : \{1,\cdots,n\} \to \Sigma_q$ is a bijection where $\Sigma_q$ is the set of zeroes of $\omega_q$ and is such that $z_q(i)$ has order $k_i$. We emphasize that isomorphisms of translation surfaces are required to respect the labeling of the zeroes. The space $\moduli$ is endowed with a natural action of $\GL$ which is a generalization of the action of $\GL$ on the space of flat tori $\GL / SL_2(\mathbb{Z})$. This action preserves the stratification of the moduli space induced by the combinatorics of the singularities, and the classification of the closed invariant sets of the strata is a central problem in Teichmüller dynamics. Such a classification has been initiated in genus $2$ by Calta in \cite{calta2004veech} and by McMullen in \cite{mcmullen2007dynamics}, \cite{mcmullen2006teichmuller} and \cite{mcmullen2005teichmuller}. In particular, the latter proved that the orbit of a genus 2 translation surface $q$ is either closed, dense in a locus of surfaces whose jacobian have a special property called real multiplication by a quadratic order, with $\omega_q$ being an eigenform or dense in its stratum and he classified the genus 2 primitive Veech surfaces. McMullen also discovered in higher genera an infinite family of non-trivial closed $\GL$-invariant sets $\Omega E_D$, parametrized by their discriminant $D$. These loci, called the Prym eigenform loci, will play a central role in the remainder of this text. The question of the classification of arithmetic (thus imprimitive) Veech surfaces in genus 2 has been addressed by Hubert and Lelièvre in \cite{hubert2004square} and later by Duryev in \cite{duryev2018teichmuller} but is still incomplete. Hubert, Lanneau and Möller have computed the orbit closure of many hyperelliptic surfaces in $\mathcal{H}(2,2)^{odd}$ in \cite{hubert2012completely}. Since, much effort has been made toward a classification in higher genera. In a celebrated result, Eskin, Mirzakhani and Mohammadi have proved a structural result on the closed invariant sets: they are immersed orbifolds cut out by linear equations with real coefficients. See Definition \ref{affinemanifold}. Such objects are called affine invariant orbifolds. Wright proved that the coefficients of the equations defining these orbifold belong to a number field $k(\M)$ whose degree over $\mathbb{Q}$ is bounded above by the genus. See \cite{wright2014field}. This number field will be referred to as the field of definition. Wright also introduced an important numerical invariant called the rank $rk(\M)$. This is a modified version of the dimension that measures the size of affine invariant orbifolds up to isoperiodic deformations and is related to cylinder deformations. See \cite{wright2015cylinder} for more details. Mirzakhani conjectured that arithmetic affine invariant orbifolds whose rank is bigger than 2 should arise from covering constructions over quadratic differentials. Arithmetic means here that the field of definition is $\mathbb{Q}$. This conjecture is now known to be false due to the work of Eskin,  McMullen, Mukamel and Wright in \cite{eskin2018billiards} but counterexamples are expected to be rare. Mirzakhani and Wright proved in \cite{mirzakhani2018full} that the only affine invariant orbifolds of maximal rank are the strata themselves and the hyperelliptic locus of those strata. Then, Apisa proved in \cite{apisa2017rank} and \cite{apisa2018gl} that the orbits of translation surfaces in the hyperelliptic strata are either closed, dense, or contained in loci of branched covers. In genus 3, Aulicino, Nguyen and Wright have classified the rank 2 affine invariant orbifolds. See \cite{nguyen2014non} and \cite{aulicino2016classification} for the classification in $\mathcal{H}^{odd}(4)$ and $\mathcal{H}^{hyp}(4)$ and see \cite{aulicino2016rank} and \cite{aulicino2016rank1} for the classification in the remaining strata of genus 3. In this paper, we pursue the classification in genus 3:

\begin{theorem}\label{T1}
Let $\mathcal{M}$ be a rel-invariant rank one affine invariant orbifold in $\mathcal{H}^{odd}(2,2)$. If $\M$ is nonarithmetic, then there exists $D$ not a square such that $\M$ is a connected component of $\Omega E_D^{odd}(2,2)$.  
\end{theorem}

Here, an affine invariant orbifold is said to be rel-invariant if it is saturated by the isoperiodic foliation (see Section 3.1 for more details). In strata with two zeroes, a rank one affine invariant orbifold is rel-invariant if, and only if, its complex dimension is $3$.  Nonarithmetic means that $k(\M)$ is strictly bigger than $\mathbb{Q}$. The sets $\Omega E_D^{odd}(2,2)$ are the intersection of the Prym eigenform loci of McMullen with the connected component $\mathcal{H}^{odd}(2,2)$. The connected components of $\Omega E_D^{odd}(2,2)$ have been classified by Lanneau and Nguyen in \cite{lanneau2014connected}. In particular, this gives an alternative proof in this setting of the fact that $k(\M)$ is totally real extension of $\mathbb{Q}$, a result otherwise due to Filip (Theorem 1.6 in \cite{filip2016semisimplicity}). In the stratum $\mathcal{H}(3,1)$, the situation is different and we prove : 

\begin{theorem}\label{T2}
Let $\M$ be a rel-invariant rank one affine invariant orbifold in $\mathcal{H}(3,1)$. Then $\M$ is arithmetic. 
\end{theorem}

\subsection{Outline of the proof}

The techniques we use are greatly inspired by the work of Apisa in \cite{apisa2017rank} and rely on cylinder deformations, a technique initiated by Wright in \cite{wright2015cylinder}. In $\mathcal{H}^{odd}(2,2)$, we introduce the notion of $\M$-entangled cylinders (see Definition \ref{entangled}) that is reminiscent of the sub-equivalent cylinders in \cite{apisa2017rank}. This notion is used to detect isometric cylinders and ultimately to produce Prym involution (see Section 2). Finally, we use a criterion to recognize Prym eigenforms due to McMullen to conclude the proof of Theorem A. Like in \cite{apisa2017rank}, an important idea of the proof is the use of vertical cylinders to control the length of horizontal saddle connections. In $\mathcal{H}(3,1)$, we use the presence of absolute cylinders and the restriction they induce on the circumferences of relative cylinders (see Proposition \ref{notmixed} for more details) to contradict nonarithmeticity, based of the formula for the field of definition of affine invariant orbifolds given by Wright, see Proposition \ref{field}.

\subsection{Organisation of the paper}

We start by recalling basic definitions for the moduli space of translation surfaces in Section \ref{framework}. In Section \ref{toolkit} we collect important results about the geometry of cylinders we will use in the course of our proofs. Section \ref{H22} is dedicated to the proof of Theorem \ref{T1}. We also deduce an interesting result about the isoperiodic leaf of nonarithmetic Veech surfaces in $\mathcal{H}^{odd}(2,2)$. Section \ref{H31} is dedicated to the proof Theorem \ref{T2}  and we prove the counterpart of the result on the isoperiodic leaf of nonarithmetic Veech surfaces in $\mathcal{H}(3,1)$. 
\subsection{Acknowledgements}

I am greatly indebted to Erwan Lanneau for his encouragements and the countless insightful comments he made during the preparation of this work. I also want to thank Paul Apisa, David Aulicino, Alex Wright and Duc-Manh Nguyen for having discussed those problems with me, this text has greatly benefited from their suggestions and remarks.  This work has been partially supported by the LabEx PERSYVAL-Lab
(ANR-11-LABX-0025-01) funded by the French program Investissement d’avenir. I am also particularly grateful to anonymous referees who made this document better in many ways. 

\section{Framework}\label{framework}

Let $S$ be a genus $g$ surface and let $\Sigma = \{z_1, \cdots, z_n\} \subset S$. The stratum $\tilde{\mathcal{H}}(\kappa)$ of the space of marked translation surfaces is the set of isomorphism classes of pairs $(q,f)$ where $q = (X_q,\omega_q,z_q)$ is a translation surface of type $\kappa$ and $f: S \to X_q$ is a homeomorphism such that for any $i \in  \{1,\cdots,n\}$ we have $f(z_i) = z_q(i)$. The following map is known as the period map : 

$$
\Phi : 
\begin{matrix}
\tilde{\mathcal{H}}(\kappa) &\to & H^1(S,\Sigma,\mathbb{C}) \\
(q,f) &\mapsto & (\gamma \mapsto \int_{f\circ \gamma} \omega_q)
\end{matrix}
$$

There is a complex structure on $\tilde{\mathcal{H}}(\kappa)$ that turns $\Phi$ into a local biholomorphism and if $MCG(S,\Sigma)$ denotes the relative mapping class group of $S$ that fixes $\Sigma$ pointwise, then $MCG(S,\Sigma)$ acts properly discontinuously on the right of $\teich$ by precomposition: $(q,f) \cdot \varphi = (q, f \circ \varphi)$. The quotient set is isomorphic to $\moduli$ and the latter is endowed with the complex orbifold structure that turns the canonical projection $\pi : \tilde{\mathcal{H}}(\kappa) \to \mathcal{H}(\kappa)$ into a local biholomorphism (in the orbifold sense, see \cite{caramello2019introduction} for relevant definitions). The space $\tilde{\mathcal{H}}(\kappa)$ is endowed with a group action by $GL_2^{+}(\mathbb{R})$ defined by : 
$$
\forall g \in \GL \ \Phi(g\cdot (q,f)) = g \cdot \Phi (q,f)
$$

That action descends to an action on $\moduli$ in a way that the canonical projection $\pi$ is $\GL$-equivariant. The action of the subgroup of diagonal matrices with determinant $1$ is known as the Teichmüller geodesic flow, while the action of the subgroup of upper triangular matrices with only 1 on the diagonal is known as the horocycle flow. More details on the structures of theses spaces can be found in \cite{zorich2006flat} or \cite{forni2013introduction}. 

\begin{dfn}\label{affinemanifold}
An affine invariant orbifold is a properly immersed closed and connected sub-orbifold which, away from self-intersection points, is cut out by linear equations in period coordinates. 
\end{dfn}

\textit{Remark.} In the previous definition, immersion refers to an immersion of orbifold, see \cite{caramello2019introduction} for relevant definitions. Most of the time, we identify an affine invariant orbifold $\M$ with its image in $\moduli$.

\medskip 

Affine invariant orbifolds are invariant under the action of $\GL$ and Eskin, Mirzakhani and Mohammadi proved that any $\GL$-orbit closure is itself an affine invariant orbifold.  See \cite{eskin2015isolation}. An important numerical invariant associated to these loci is the rank, defined as follows: define $\rho :  H^1(S,\Sigma,\mathbb{C}) \to H^1(S,\mathbb{C})$ to be the canonical restriction map. By definition, the tangent space at a point $q \in \M$ is identified to a linear subspace $V \subset H^1(S,\Sigma,\mathbb{C})$. Avila, Eskin and Möller proved in \cite{avila2017symplectic} that $\rho(V)$ is a symplectic subspace of $H^1(S,\mathbb{C})$. The rank of $\M$, denoted $rk(\M)$, is then defined as half the dimension of this space. Notice that this dimension is constant by connectedness. More details can be found in \cite{wright2015cylinder}. The following definition, also due to Wright, will be important for the remainder of this text :

\begin{dfn}
Let $\M$ be a affine invariant orbifold. The field of definition of $\M$ is the smallest subfield $k(\M)$ whose coefficients can be used to define $\M$ by linear equations in period coordinates. 
\end{dfn}

A particularly interesting family of rank 1 affine invariant orbifolds has been discovered by McMullen in \cite{mcmullen2007prym}. We recall here the definition. Let $q = (X,\omega)$ be a translation surface endowed with a holomorphic involution $\lambda$. We denote by $\Omega(X)$ the set of holomorphic $1$-forms and by $\Omega^-(X)$ the set of $\lambda$-anti-invariant holomorphic $1$-forms. We say that $q$ is a Prym form if $\omega \in \Omega^-(X)$, that is $\lambda^{\ast} \omega = - \omega$, and $dim \Omega^-(X) = 2$. The Prym variety $Prym(q,\lambda)$ is defined as the 2-dimensional abelian variety $(\Omega^-(X))^{\ast} / H_1^-(X,\mathbb{Z})$ endowed with the polarization coming from the intersection form on $H_1(X,\mathbb{Z})$. Finally, let $D$ be a positive integer congruent to 0 or 1 mod 4 and let $\mathcal{O}_D \simeq \mathbb{Z}[X]/(X^2+bX+c)$ be the real quadratic order of discriminant $D = b^2-4c$. 

\begin{dfn}
A Prym eigenform is a Prym form $q = (X,\omega)$ corresponding to an involution $\lambda$ such that there is an injective ring morphism $\mathfrak{i} : \mathcal{O}_D \to End(Prym(q,\lambda))$ satisfying the following properties:

\begin{enumerate}
\item $\mathfrak{i}(\mathcal{O}_D)$ is a proper subring comprised of self adjoint endomorphisms
\item $\omega$ is an eigenvector for the action of $\mathcal{O}_D$ on $\Omega(X)^{-}$
\end{enumerate}
\end{dfn}

We will denote by $\Omega E_D(\kappa)$ the set of Prym eigenforms in the stratum $\moduli$. For more details, see \cite{mcmullen2007prym}. The following is due to McMullen (Theorem 3.2 in \cite{mcmullen2007prym})

\begin{thm}
Any connected component of $\Omega E_D(\kappa)$ is a rank 1 affine invariant orbifold.
\end{thm}

When $\kappa = (2,2)^{odd}$ the loci $\Omega E_D(\kappa)$ are rel-invariant and Theorem A states that any nonarithmetic rank 1 and rel-invariant affine invariant orbifold is of that form. Finally, note that $\Omega E_D(3,1)$ is empty. Indeed if $q= (X,\omega)$ is Prym form in $\mathcal{H}(3,1)$ corresponding to a involution $\lambda$, then $\omega^2$ is invariant by $\lambda$ and thus descends to a quadratic form  on the torus whose singularities counted with multiplicity do not add up to $0$, as they should according to Riemann-Roch.

\section{Preparation of a toolkit}\label{toolkit}

In this section, we collect the tools we will use in the proof of Theorems A and B. 
 
\subsection{Isoperiodic foliation and Rel flow}

The stratum $\moduli$ is endowed with a foliation $\mathcal{F}$ that we recall. Since $\Phi$ is a local biholomorphism and $\rho$ is a surjective linear map, the composition $\rho \circ \Phi$ is a submersion. The connected components of its level sets induce a foliation of $\teich$ whose plaques are modeled on $ker \rho$. The group $MCG(S,\Sigma)$ acts on both $\teich$ by precomposition on the marking and on $H^1(S,\mathbb{C})$ by the Torelli representation. With respect to these actions, the map $\rho \circ \Phi$ is $MCG(S,\Sigma)$-equivariant. This means that the aforementioned foliation on $\teich$ descends to a foliation on $\moduli$. This foliation is usually referred to as the isoperiodic foliation, the kernel foliation or the Rel foliation. Intuitively, moving along the leaves of the isoperiodic foliations amounts to changing the relative periods without changing the absolute ones. In the same fashion, one can also define vector fields that are tangent to the leaves of the isoperiodic foliations. Any $v \in ker \rho$ induces by pull back by $\Phi$ a constant vector field on $\teich$. This vector field is invariant by the action of $MCG(S,\Sigma)$ (because we required that labeling of zeroes is respected by $MCG(S,\Sigma)$) and thus descends to a vector field on $\moduli$. We denote by $Rel_v^t$ the flow of that vector field. The latter is known as the rel flow associated to $v$. If $v \in H^1(S,\Sigma,\mathbb{R})$ (resp. $H^1(S,\Sigma,i\cdot \mathbb{R})$ ) we say that $Rel_v^t$ is a real rel flow (resp. imaginary rel flow). It is known that the presence of horizontal (resp. vertical) saddle connections connecting different zeroes is the only obstruction for real (resp imaginary) rel flows to be defined for all time, see Section 4 of \cite{bainbridge2016horocycle} for more details. 

\begin{dfn}
Let $\M \subset \moduli$ be an affine invariant orbifold. We will say that $\M$ is rel-invariant if it is saturated by the isoperiodic foliation. Equivalently, this means that for any $q \in \M$, any $v \in ker\rho$ and any $t>0$ such that $Rel_v^t(q)$ is defined, we have $Rel_v^t(q) \in \M$.
\end{dfn}

As mentioned in the introduction, the space $\Omega E_D^{odd}(2,2)$ is rel-invariant. Closed $\GL$-orbits and hyperelliptic loci in nonhyperelliptic strata are never Rel-invariant. More sophisticated examples of $\M$ that are not rel-invariant are provided in \cite{eskin2018billiards}. 

\subsection{Modifying the twist parameters}

Let $q$ be a translation surface in $\moduli$ that has a decomposition into $m$ horizontal cylinders, which we denote by $\mathcal{C}_1, \cdots, \mathcal{C}_m$. We also denote by $c_i$ the circumference of the cylinder $\mathcal{C}_i$ and by $h_i$ its height. One can define a smooth map (in the orbifold category) $\tau$ from $\mathbb{R}^m$ to $\moduli$ such that $\tau(0) = q$ and $\tau(x_1,\cdots,x_m)$ is the surface obtained from $q$ by shearing the cylinder $\mathcal{C}_i$ by $x_ic_i$. The action of $\mathbb{Z}^m$ on $\mathbb{R}^m$ by translation is intertwined by $\tau$ with the action of the subgroup of $MCG(S,\Sigma)$ spanned by the Dehn twists about the core curves of the horizontal cylinders. There is also a finite subgroup $\Delta$, perhaps trivial, of automorphisms of $q$ that acts on $q$ by permuting cylinders and on the $m$-dimensional torus $\mathbb{T}^m$ by permuting the coordinates. Consequently, there is an induced map $\tau : \mathbb{T}^m / \Delta \to \moduli$ and it is a diffeomorphism (in the orbifold category) onto its image. See Figure \ref{twistmap}. We have the following formula for the image of a vector $(u_1,\cdots,u_m) \in \mathbb{R}^m$ by the derivative of $\tau$ at a point $x \in \mathbb{R}^m$: 

\begin{equation}\label{derivative}
    D_x \tau(u_1,\cdots,u_m) = \sum_{i=1}^m c_iu_i I(\gamma_i, \cdot)
\end{equation}

where $I : H_1(S-\Sigma,\mathbb{Z}) \times H_1(S,\Sigma,\mathbb{Z}) \to \mathbb{Z}$ is the intersection product and $\gamma_i$ is the core curve of the cylinder $\mathcal{C}_i$. See Sections 6.1 and 6.2 in \cite{hooper2015rel} for more details.

\medskip

\begin{figure}[h]
    \centering
\begin{tikzpicture}
    \begin{scope}[shift={(0,0)}, scale=0.7]
    \draw (1,0) rectangle (2,0.5);
    \draw (0,.5) rectangle (2,1);
    \draw (0,1) rectangle (1,1.5);
    \draw (1,-.5) node{$\tau(0) = q$};
    \draw[>=latex,->](2.5,.75) .. controls +(45:.5) and +(135:.5) .. (4.5,.75) ;
    \end{scope}
    
    \begin{scope}[shift={(4,0)},scale=0.7]
    \draw (1,0.5) -- (.5,0) -- (1.5,0) -- (2,.5);
    \draw (0,.5) rectangle (2,1);
    \draw (0,1) -- (1,1.5) -- (2,1.5) -- (1,1);
    \draw (1,-.5) node{$\tau(1,0,-1/2)$};
    \draw[fill=black,opacity = .1] (1,1) -- (1,1.5) -- (2,1.5) -- cycle;
    \draw[dashed] (1,1) -- (1,1.5);
    \end{scope}
    
    \begin{scope}[shift={(7,0)},scale=0.7]
    \draw (-1,0.75) node{$=$};
    \draw (1,0.5) -- (.5,0) -- (1.5,0) -- (2,.5);
    \draw (0,.5) rectangle (2,1);
    \draw (0,1) rectangle (1,1.5);
    \draw (1,-.5) node{$\tau (0,0,-1/2)$};
    \draw[fill=black,opacity = .1] (0,1) -- (0,1.5) -- (1,1.5) -- cycle;
    \draw[dashed] (0,1) -- (1,1.5);
    \end{scope}

    \end{tikzpicture}
	 \caption{twisting cylinders}
    \label{twistmap}
\end{figure}

\medskip

From now on and until the end of this section we assume that \textbf{$\moduli$ is a stratum of surfaces with exactly two singularities and that the horizontal decomposition of $q$ is stable \textit{i.e} any horizontal saddle connection connects a zero to itself.}  We shall also say that a cylinder is stable when any saddle connection on one of its boundary components connects a zero to itself, hence the horizontal decomposition is stable if, and only if, the horizontal cylinders are stable. A stable cylinder is said to be relative if the singularity on its top component is not the same as the one on its bottom component, or equivalently, if any cross section of this cylinder is not an absolute cycle. We recall that a cross section of a cylinder is a path inside that cylinder that connects a singularity on one of its boundary component to a singularity on the other boundary component, while crossing the core curve only once. Conversely, a cylinder is said to be absolute if it has the same singularities on its two boundaries. There is an equivalence relation on the set of relative cylinders, defined as follows. Denote by $\partial$ the boundary map from $H_1(X,\Sigma)$ to $H_0(\Sigma)$ and choose for each cylinder $\mathcal{C}_i$ a cross section $\beta_i$. Since we assumed that there are only two singularities, the image of $\partial$ is one-dimensional. Thus $\mathcal{C}_1$ and $\mathcal{C}_2$ are said to be of same type if $\partial(\beta_1)$ is positively colinear to $\partial(\beta_2)$. Notice that this does not depend on the choice of $\beta_1$ and $\beta_2$ as we assumed that the horizontal cylinders are stable. We denote by $\mathfrak{C}^+$ and $\mathfrak{C}^-$ the two families of relative cylinders having the same type. Note that the absolute cylinders are the ones whose cross sections are in the kernel of $\partial$ and we denote by $\mathfrak{C}^0$ the set of absolute cylinders. Let $\delta_i$ be $1$ if $\mathcal{C}_i \in \mathfrak{C}^+$, $-1$ if $\mathcal{C}_i \in \mathfrak{C}^-$ and $0$ if $\mathcal{C}_i$ is an absolute cylinder. For a given cylinder $\mathcal{C}_i$, we denote by $\mu_i = h_ic_i^{-1}$ its modulus and let $\mu := (\mu_1,\cdots,\mu_m)$ , $u = (\delta_1c_1^{-1},\cdots,\delta_mc_m^{-1})$, $V_q = span_{\mathbb{R}} \langle \mu,u \rangle \subset \mathbb{R}^m$. Let $d$ denote the algebraic degree of the $\delta_ic^{-1}_i$, \textit{i.e} the dimension of the $\mathbb{Q}$-vector space spanned by these numbers. Finally, let $z \in S$ corresponding to the singularity located on the bottom component of cylinders in $\mathfrak{C}^+$ and let $\gamma$ be a homotopically trivial loop on $S$ that circles $z$ clockwise and let $v = I(\gamma,\cdot) \in ker \rho$. We recall that $I : H_1(S-\Sigma,\mathbb{Z}) \times H_1(S,\Sigma,\mathbb{Z}) \to \mathbb{Z}$ is the intersection product. Notice that by definition, we have for any $i \in \{1,\cdots,m\}$  $v(\beta_i) = \delta_i$. Finally, for any $x \in \mathbb{R}^m$ and $t>0$ the surface $Rel_v^t \circ \tau(x)$ is well-defined, which comes from the fact that $\tau(x)$ does not have horizontal saddle connections connecting different singularities.  

\begin{prop}\label{intertwinedflows}
For any $w = t_1\mu + t_2u \in V_q$, the straight line flow $\phi_w^t : x \mapsto x + t\cdot w$ on $\mathbb{R}^m$ induces a well-defined flow on $\mathbb{T}^m / \Delta$ and for any $t>0$, we have:

\begin{equation}\label{intertwined}
\tau \circ \phi_w^t = \left( \begin{smallmatrix} 1 && t \cdot t_1 \\ 0  && 1 \end{smallmatrix} \right) \cdot Rel_{t_2\cdot v}^t  \circ \tau
\end{equation}

\end{prop}

\begin{proof} We first show that the straight line flow $\phi_w^t$ induces a well-defined flow on $\mathbb{T}^m / \Delta$. It is enough to show that $\phi_w^t$ commutes with the action of $\Delta$, which amounts to showing that if $\Delta$ permutes $\mathcal{C}_i$ with $\mathcal{C}_j$ then $\mu_i = \mu_j$ and $u_i = u_j$. But because $\Delta$ can only swap isometric cylinders, we necessarily have $c_i = c_j$ and $h_i = h_j$. This shows that $\mu_i = \mu_j$. Then, since zeroes are labeled, we also know that $\delta_i = \delta_j$ and that shows that $u_i = u_j$. Now, using Equation \eqref{derivative}, we know that for any $x \in \mathbb{R}^m$, we have $D_x \tau(u) = \sum_{i=1}^m \delta_i I(\gamma_i, \cdot)$ and we claim that this is exactly $v$. One way to see this is to use a basis of $H_1(S,\Sigma,\mathbb{C})$ obtained by adjoining to the $\beta_i$ a family of saddle connections that live on the boundaries of the horizontal cylinders of $q$. Notice that $v$ evaluates to zero on the latter because we assumed that the decomposition of $q$ into horizontal cylinders is stable, which means that the boundary components of the cylinders of $q$ are made of absolute cycles. By definition, $v$ vanishes on absolute cycles. Thus, after noticing that the the $\beta_i$ are dual to the $\gamma_i$ with respect to the intersection product $I$, we see that $v = \sum_{i=1}^m v(\beta_i)I(\gamma_i,\cdot)$ and by definition $v(\beta_i) = \delta_i$. This computation shows that the constant vector field $w$ on $\mathbb{R}^m$, whose associated flow is $\phi_w^t$, is mapped by $\tau$ to the vector field that defines $Rel_v^t$. This proves $\tau \circ \phi_u^t =  Rel_{u}^t  \circ \tau$.  We show in the same fashion that $\tau \circ \phi_{\mu}^t = \left( \begin{smallmatrix} 1 && t \\ 0  && 1 \end{smallmatrix} \right) \cdot \tau$ and we conclude using the fact that real rel flows commute with the action of the upper triangular subgroup of $\GL$. 
\end{proof}

We shall say that a $\mathbb{R}$-linear subspace of $\mathbb{R}^m$ is rational if it has a basis comprised of vectors with rational entries, or equivalently that it is defined by linear equations with rational coefficients.

\begin{prop}\label{twist}
If $q$ is contained in a rel-invariant rank one affine invariant orbifold $\M$, then $V_q$ is rational. In particular, $d \leq 2$.  
\end{prop}

\begin{proof}

Denote by $pr : \mathbb{R}^m \to \mathbb{T}^m / \Delta$ the canonical projection. By assumption $\M$ is rel-invariant and $\GL$-invariant, it thus follows from Equation \eqref{intertwined} that $\tau \circ pr(V_q)$ is contained in $\M$. Let $V$ be the smallest rational subspace that contains $V_q$. It is a consequence of the Kronecker theorem (see Proposition 7 on page 74 of \cite{bourbaki1966elements}) that $pr(V_q)$ is dense in $pr(V)$ and since $\M$ is closed, we deduce that $\tau \circ pr(V)$ is contained in $\M$. In particular, $D_0\tau(V)\otimes \mathbb{C}$ is a complex linear subspace of $T_q\M$ that we view as a subspace of $H^1(S,\Sigma,\mathbb{C})$. It follows from Equation $\eqref{derivative}$ together with the fact that core curves of parallel cylinders do not intersect that $\rho(D_0\tau(V)\otimes \mathbb{C})$ is an isotropic subspace of $H^1(S,\mathbb{C})$. But since $\M$ has rank 1, we know that $\rho(T_q\M)$ is a symplectic subspace of complex dimension $2$. It comes that the complex dimension of $\rho(D_0\tau(V)\otimes \mathbb{C})$ is 1 and since $ker \rho$ has complex dimension $1$, that $D_0\tau(V)\otimes \mathbb{C}$ has complex dimension at most $2$. Using that $D_0\tau$ is injective, we deduce that $V$ has real dimension at most 2 and thus that $V_q = V$ \textit{i.e} that $V_q$ is rational. 
\end{proof}

\begin{prop}\label{equation}
If $d = 2$, any linear equation with rational coefficients satisfied by the entries of $u$ (\textit{i.e} by the $\delta_i c^{-1}_i$) are also satisfied by the entries of the elements of $V_q$. 
\end{prop}

\begin{proof}
By Proposition \ref{twist}, the smallest rational linear space that contains $u$ is $V_q$. In particular, any linear relation with rational coefficients satisfied by the entries of $u$ has to be satisfied by that of any vector in $V_q$.
\end{proof}

Proposition \ref{equation} also appears in \cite{apisa2017rank} (Lemma 3.4) in a slightly different disguise.

\begin{cor}\label{notmixed}
If the relative cylinders do not cover $X_q$ then their circumferences are pairwise commensurable.
\end{cor}

\begin{proof}
Let $\mathcal{C}_i$ be one of the horizontal absolute cylinders and suppose there are relative cylinders with non commensurable circumferences. This means in particular that $d = 2$. By definition, we have $\delta_i=0$ and thus the $ith$ entry of $u$ vanishes. This is a linear relation with rational coefficients satisfied by the entries of $u$ and thus by Proposition \ref{equation}, this equation must be satisfied by the entries of $\mu$. This is a contradiction as moduli do not vanish.  
\end{proof}

\begin{prop}\label{noneq}
If $d = 2$, two relatives cylinders that do not have same type cannot have commensurable circumferences.  
\end{prop}

\begin{proof}
Suppose to a contradiction that two cylinders $\mathcal{C}_1$ and $\mathcal{C}_2$ that do not have same type have commensurable circumferences, and say $\mathcal{C}_1$ belongs to $\mathfrak{C}^+$. There is thus a positive rational number $r$ such that $c_1^{-1} = r c_2^{-1} = -r \delta_2c_2^{-1}$. By Proposition \ref{equation}, it implies that $\mu_1 = - r \mu_2$. This is a contradiction as moduli are positive numbers. 
\end{proof}

\begin{cor}\label{height}
If $d = 2$, any pair of relative cylinders with commensurable circumferences have same height. Reciprocally, if two cylinders of same type have same height, then their circumferences are commensurable.
\end{cor}

\begin{proof}
For the first assertion let $\mathcal{C}_1$ and $\mathcal{C}_2$ be two relative cylinders with commensurable circumferences. Proposition \ref{noneq} implies that these cylinders have same type, \textit{i.e} $\delta_1 = \delta_2$. Take a rational number $r$ such that $c^{-1}_1 = rc^{-1}_2$. By Proposition \ref{equation}, it implies that $\mu_1 = r \mu_2 = \frac{c_2}{c_1} \mu_2$. This equation is exactly $h_1 = h_2$. \\

For the second assertion, let $\mathcal{C}_1$ and $\mathcal{C}_2$ be two relative cylinders with same type and same height $h$ and suppose their circumferences are not commensurable. Let $\mathcal{C}_3$ be another relative cylinder of the other type, which means that $\delta_1 \delta_3 = \delta_2 \delta_3 = -1$. Since $d = 2$, we can find $r,s \in \mathbb{Q}$ such that $\delta_3c^{-1}_3 =r \delta_1c^{-1}_1 + s \delta_2 c^{-1}_2$. Proposition \ref{equation} now implies that $\mu_3 =
r \mu_1 + s \mu_2$ and this relation simplifies to $h_3 = \delta_1 \delta_3 h = -h$ which is a contradiction.
\end{proof}

\subsection{The field of definition}

We recall a useful formula for the field of definition of an affine invariant orbifold $\M$. It has been proved by Wright in \cite{wright2015cylinder}. 

\begin{prop}\label{field}
Let $q$ be a translation surface in $\M$ that is decomposed into $m$ horizontal cylinders whose circumferences are denoted by $c_i$. Then, the following formula holds : 

$$
k(\M) \subseteq \mathbb{Q}[c_2c_1^{-1},\cdots,c_mc_1^{-1}]
$$
\end{prop}

Wright actually proved a stronger version and established the other inclusion if one considers only a subclass of cylinders, but we will not need such generality. 

\subsection{Cylinders, complete periodicity and overcollapse}\label{overcollapse}

\begin{prop}\label{existence}
There exists a horizontally periodic surface in any affine invariant orbifold. 
\end{prop}

\begin{proof}
Let $\M$ be an affine invariant orbifold and let $q \in \M$. Corollary 6 in \cite{smillie2004minimal} states that there is a horizontally periodic surface in the horocycle orbit closure of $q$. The proposition follows from the fact that $\M$ is $GL_2(\mathbb{R})$-invariant and closed.  
\end{proof}

The last tool we will need is the complete periodicity property. See Proposition 4.20 and Remark 4.21 in \cite{wright2014translation}

\begin{prop}\label{periodicity}
Let $\M$ be a rank one affine invariant orbifold and let $q$ be a translation surface in $\M$. If there is a saddle connection on $X_q$ in direction $\theta$ that joins a singularity to itself, then $X_q$ is decomposed into cylinders in direction $\theta$. 
\end{prop}

We end this section by recalling a particular case of a construction that already appeared in the work of Apisa, Aulicino, Nguyen and Wright. It is usually referred to as the overcollapse of a cylinder or an extended cylinder deformation. See \cite{apisa2017rank}, \cite{aulicino2016rank1}, \cite{aulicino2016rank}, \cite{apisa2021marked}, \cite{apisa2021high}, \cite{apisa2020reconstructing}. Let $q \in \moduli$ be a horizontally periodic translation surface with a stable decomposition, let $h$ be the smallest height of a cylinder in $\mathfrak{C}^-$ and let $\mathcal{C}$ be a cylinder in $\mathfrak{C}^-$ whose height is $h$. Let also $\varepsilon$ be any positive number small enough so that cylinders in $\mathfrak{C}^-$ either have their height equal to $h$ or bigger than $h+\varepsilon$. We can shear the surface $q$ using the horocycle flow so that $q$ does not have any vertical saddle connection and let $q' = Rel_{i\cdot v}^{h+\varepsilon}(q)$. We say that $q'$ is obtained by overcollapsing the cylinder $\mathcal{C}$ using rel, or just overcollapsing the cylinder $\mathcal{C}$ for simplicity. We recall that $v$ was defined in Section 3.2, just before Proposition \ref{intertwinedflows}. Notice that the path $\gamma: t \mapsto Rel_{i \cdot v}^t(q)$ connects $q$ to $q'$ and for $t \in [0,h)$, the surface $\gamma(t)$ is obtained from $q$ by changing the height of the cylinders: for the ones in $\mathfrak{C}^+$ it is increased by $t$, for the ones in $\mathfrak{C}^-$ it is decreased by $t$ and for the ones in $\mathfrak{C}^0$ it is unchanged. When $t$ tends to $h$, the height of the cylinder $\mathcal{C}$ then reaches $0$, we say it collapses. This is also the case of all the other cylinders in $\mathfrak{C}^-$ whose height is $h$. Just after $t=h$, we see the birth of one, or several, cylinders of height $t-h$ hence the term overcollapsing. We include a depiction of that phenomenon: 

\begin{center}
    \begin{tikzpicture}
    \begin{scope}[scale=.8]
    \draw (0,.5) -- (0,0) -- (2,0) -- (2,.5) -- (1,.5);
    \draw (0,.5) -- (.5,1) -- (1.5,1) -- (1,.5);
    
    \draw (0,0) node{$\times$};
    \draw (1,0) node{$\times$};
    \draw (2,0) node{$\times$};
    \draw (.5,1) node{$\times$};
    \draw (1.5,1) node{$\times$};
    
    \draw (0,.5) node{$\bullet$};
    \draw (1,.5) node{$\bullet$};
    \draw (2,.5) node{$\bullet$};
    
    \draw (.5,-.25) node[scale = .7]{A};
    \draw (1,1.25) node[scale = .7]{A};
    
    \draw (.6,.75) node[scale = .8]{$\mathcal{C}$};
    
    \draw (1,-1) node{$q = \gamma(0)$};

    \end{scope}
   
    \begin{scope}[scale=.8, shift = {(5,0)}]
    \draw (0,0) rectangle (2,1); 
    
    \draw (0,0) node{$\times$};
    \draw (1,0) node{$\times$};
    \draw (2,0) node{$\times$};
    \draw (.5,1) node{$\times$};
    
    \draw (0,1) node{$\bullet$};
    \draw (.5,0) node{$\bullet$};
    \draw (1,1) node{$\bullet$};    
    \draw (2,1) node{$\bullet$};

    \draw (1,-1) node{$\gamma(h)$};
    \end{scope}
    
    \begin{scope}[scale=.8, shift = {(10,0)}]
    \draw (0,0) -- (0,1.25); 
    \draw (1,1.25) -- (2,1.25) -- (2,0) -- (1,0) ;
    \draw (0,0) -- (.5,.25) -- (1,0); 
    \draw (0,1.25) -- (.5,1) -- (1,1.25);
    \draw (0,.25) -- (2,.25);
    \draw (0,1) -- (2,1);
    
    \draw (0,0) node{$\times$};
    \draw (1,0) node{$\times$};
    \draw (2,0) node{$\times$};
    \draw (.5,1) node{$\times$};
    
    \draw (0,1.25) node{$\bullet$};
    \draw (1,1.25) node{$\bullet$};
    \draw (2,1.25) node{$\bullet$};
    \draw (.5,.25) node{$\bullet$};   
    
    \draw (1,-1) node{$q' = \gamma(h+\varepsilon)$}; 
    \end{scope}
    \end{tikzpicture}
\end{center}

\section{The stratum $\mathcal{H}^{odd}(2,2)$} \label{H22}

\begin{dfn}\label{entangled}
Let $\M$ be an affine invariant orbifold and let $q \in \M$. We say that two cylinders on $q$ are $\M$-entangled if they are parallel, have same height and remain so on any nearby surface in $\M$.
\end{dfn}

In the sequel, we will use the list of cylinder diagrams in $\mathcal{H}^{odd}(2,2)$ given in Appendix A. When we say that the cylinder diagram of $q$ is given by $i.j$, we mean that $q$ is horizontally periodic and that the cylinder diagram of $q$ is the one depicted in figure $i.j$ in Appendix A. We use the convention that cylinders are numbered from bottom to top. 

\begin{prop}\label{atmost}
Let $\M$ be a rel-invariant and nonarithmetic affine invariant orbifold. There are at most three families of $\M$-entangled cylinders in a given direction on any surface $q \in \M$. In particular, if the cylinder diagram of $q$ is given by 4.2, then two of the cylinders amongst $\mathcal{C}_2$, $\mathcal{C}_3$ and $\mathcal{C}_4$ are $\M$-entangled and if it is given by 4.3, then $\mathcal{C}_1$ and $\mathcal{C}_4$ are $\M$-entangled. Case 4.1 does not occur and Cases 4.5 and 4.4 can be reduced to Cases 4.2 and 4.3 (via overcollapse).
\end{prop}

\begin{proof}
The proof will use overcollapses of cylinders, see \S \ref{overcollapse} for a definition. Let $q \in \M$ and suppose $q$ has a horizontal cylinder. By Proposition \ref{periodicity}, we know that $q$ is horizontally periodic. If the corresponding cylinder decomposition is made of at most three cylinders there is nothing to prove so we can assume that $q$ is made of four horizontal cylinders. We analyze the different possible cases for the cylinder diagram of $q$ separately:

\textbf{Case 4.2} We claim that two of the cylinders 2, 3 and 4 are $\M$-entangled. Label the types of cylinders so that these cylinders are in $\mathfrak{C}^-$. If there are two cylinders in $\mathfrak{C}^-$ with same height, these cylinders will remain parallel on any nearby surface since the rank of $\M$ is 1 and because they have the same type they also keep the same height on any nearby surface, thus they are $\M$-entangled. We recall that a neighborhood of $q$ when $\M$ has rank 1 is described by surfaces of the form $g \cdot Rel_w^1(q)$, where $w$ is any small vector in $ker\rho$ and $g$ is a matrix in $\GL$ close to the identity. Assume to a contradiction that no two cylinders in $\mathfrak{C}^-$ have same height and we can assume without loss of generality that the cylinder with smallest height is $\mathcal{C}_4$. Shear the surface $q$ so that $\mathcal{C}_4$ is not crossed by a vertical saddle connection and overcollapse it using rel to reach a surface $q'$ whose cylinder decomposition is depicted in the following figure:

\begin{center}
    \begin{tikzpicture}
    \begin{scope}[scale=.6]

\draw (0,0) rectangle (3,1);
\draw (0,2) rectangle (1,3);
\draw (0,4) rectangle (1,5);
\draw (0,6) -- (.5,7) -- (1.5,7) -- (1,6) -- cycle;

\draw (0,0) node{$\bullet$}; 
\draw (1,0) node{$\bullet$}; 
\draw (2,0) node{$\bullet$}; 
\draw (3,0) node{$\bullet$}; 

\draw (0,1) node{$\times$}; 
\draw (1,1) node{$\times$}; 
\draw (2,1) node{$\times$}; 
\draw (3,1) node{$\times$}; 

\draw (0,2) node{$\times$};
\draw (1,2) node{$\times$};

\draw (0,3) node{$\bullet$};
\draw (1,3) node{$\bullet$};

\draw (0,4) node{$\times$};
\draw (1,4) node{$\times$};

\draw (0,5) node{$\bullet$};
\draw (1,5) node{$\bullet$};

\draw (0,6) node{$\times$};
\draw (1,6) node{$\times$};

\draw (0.5,7) node{$\bullet$};
\draw (1.5,7) node{$\bullet$};

\draw (.5,.25) node[scale=.7]{0};
\draw (1.5,.25) node[scale=.7]{2};
\draw (2.5,.25) node[scale=.7]{1};

\draw (.5,.75) node[scale=.7]{5};
\draw (1.5,.75) node[scale=.7]{4};
\draw (2.5,.75) node[scale=.7]{3};

\draw (.5,2.25) node[scale=.7]{3};
\draw (.5,2.75) node[scale=.7]{1};

\draw (.5,4.25) node[scale=.7]{4};
\draw (.5,4.75) node[scale=.7]{2};

\draw (.5,6.25) node[scale=.7]{5};
\draw (1,6.75) node[scale=.7]{0};

\draw (1.5,-1) node{$q$};
\draw[>=latex,->](4,2.75) -- (5.5,2.75) ;
\draw (4.75,3.4) node[scale=.6]{overcollapse $\mathcal{C}_4$};

\end{scope}

 \begin{scope}[shift={(5,0)},scale=.6]

\draw (0,0) -- (.5,.25) -- (1,0) -- (3,0) -- (3,1.5) -- (1,1.5) -- (.5,1.25) -- (0,1.5) -- cycle;
\draw (0,2.5) rectangle (1,3);
\draw (0,4) rectangle (1,4.5);

\draw (0,0) node{$\bullet$}; 
\draw (1,0) node{$\bullet$}; 
\draw (2,0) node{$\bullet$}; 
\draw (3,0) node{$\bullet$}; 
\draw (.5,1.25) node{$\bullet$};

\draw (0,1.5) node{$\times$}; 
\draw (1,1.5) node{$\times$}; 
\draw (2,1.5) node{$\times$}; 
\draw (3,1.5) node{$\times$};
\draw (.5,.25) node{$\times$};

\draw (0,2.5) node{$\times$};
\draw (1,2.5) node{$\times$};

\draw (0,3) node{$\bullet$};
\draw (1,3) node{$\bullet$};

\draw (0,4) node{$\times$};
\draw (1,4) node{$\times$};

\draw (0,4.5) node{$\bullet$};
\draw (1,4.5) node{$\bullet$};

\draw (1.5,.25) node[scale=.7]{2};
\draw (2.5,.25) node[scale=.7]{1};

\draw (1.5,1.25) node[scale=.7]{4};
\draw (2.5,1.25) node[scale=.7]{3};

\draw (.5,2.25) node[scale=.7]{3};
\draw (.5,3.25) node[scale=.7]{1};

\draw (.5,3.75) node[scale=.7]{4};
\draw (.5,4.75) node[scale=.7]{2};

\draw (1.5,-1) node{$q'$};

\end{scope}
    \end{tikzpicture}
\end{center}

The surface $q'$ has three cylinders $\mathcal{C}'_1$, $\mathcal{C}'_2$ and $\mathcal{C}'_3$ coming from the cylinders $\mathcal{C}_1$, $\mathcal{C}_2$ and $\mathcal{C}_3$ on $q$ plus an additional cylinder $\mathcal{C}$ of height $\varepsilon$ (we use the notation of \S \ref{overcollapse}) that was created by the overcollapse of $\mathcal{C}_4$. Notice that for any $i \in \{1,2,3\}$, the circumference of the cylinder $\mathcal{C}'_i$ is the same as the one of $\mathcal{C}_i$, this is because our overcollapse deformation is isoperiodic and the circumference of a cylinder is the period of its core curve. The cylinders $\mathcal{C}_2$ and $\mathcal{C}_3$ do not have commensurable circumferences by Propositions \ref{height} and \ref{field} and thus there are two rational numbers $r,s$ such that $c^{-1}_1 = rc^{-1}_2 + sc^{-1}_3$, where $c_i$ is the circumference of $\mathcal{C}_i$. Denote by $\mu_1,\mu_2$ and $\mu_3$ (resp. $\mu'_1,\mu'_2$ and $\mu'_3$) the moduli of the cylinders $\mathcal{C}_1, \mathcal{C}_2$ and $\mathcal{C}_3$ on $q$ (resp. $\mathcal{C}'_1, \mathcal{C}'_2$ and $\mathcal{C}'_3$ on $q'$). The cylinder decompositions on $q$ and $q'$ are made of relative cylinders and their circumferences are not all pairwise commensurable by Proposition \ref{field}. It thus comes from Proposition \ref{equation} that:

\begin{align*}
    -\mu_1 &= r\mu_2 + s\mu_3 \\
    \mu'_1 &= r\mu'_2 + s\mu'_3 
\end{align*}

The important thing to notice here is  the minus sign in front of $\mu_1$ that does not appear in front of $\mu_1'$. This is due to that fact that when $\mathcal{C}_4$ overcollapses, the type of $\mathcal{C}_1$ is changed. We also compute that $\mu_1' = \mu_1 + c^{-1}_1(h_4 - \varepsilon)$, $\mu_2' = \mu_2 - c^{-1}_2(h_4+\varepsilon)$ and $\mu_3' = \mu_3 - c^{-1}_3(h_4+\varepsilon)$. We thus get: 

\begin{align*}
\mu'_1 - \mu_1 &= r(\mu_2 + \mu_2') + s(\mu_3 + \mu_3') \\
    &= 2r\mu_2 + 2s\mu_3 - (h_4 + \varepsilon)(rc^{-1}_2 + sc^{-1}_3) \\
    &= -2\mu_1 - c^{-1}_1(h_4 + \varepsilon) \\
\end{align*}

Using that $\mu_1' - \mu_1 = c^{-1}_1(h_4 - \varepsilon)$, we obtain that $h_1 + h_4 = 0$ which is a contradiction as the height of a cylinder is a strictly positive number. It follows that at least two of the cylinders $\mathcal{C}_2$, $\mathcal{C}_3$ and $\mathcal{C}_4$ have same height and these two cylinders are $\M$-entangled. 

\textbf{Case 4.3} We claim that $\mathcal{C}_1$ and $\mathcal{C}_4$ are $\M$-entangled. Indeed, suppose to a contradiction that $h_1 > h_4$ and shear the surface $q$ such that the saddle connection labeled $2$ sits above the singularity $\bullet$. If we overcollapse the cylinder $\mathcal{C}_4$, we reach a surface $q'$ as depicted below:

\begin{center}
    \begin{tikzpicture}
    \begin{scope}[scale=0.6]

\draw (0,0) rectangle (2,1);
\draw (0,2) rectangle (1,3);
\draw (0,4) rectangle (3,5);
\draw (0,6) -- (-.5,7) -- (1.5,7) -- (2,6) -- cycle;

\draw (0,0) node{$\bullet$};
\draw (1,0) node{$\bullet$};
\draw (2,0) node{$\bullet$};

\draw (0,1) node{$\times$};
\draw (2,1) node{$\times$};

\draw (0,2) node{$\times$};
\draw (1,2) node{$\times$};

\draw (0,3) node{$\bullet$};
\draw (1,3) node{$\bullet$};

\draw (0,4) node{$\times$};
\draw (1,4) node{$\times$};
\draw (3,4) node{$\times$};

\draw (0,5) node{$\bullet$};
\draw (2,5) node{$\bullet$};
\draw (3,5) node{$\bullet$};

\draw (0,6) node{$\bullet$};
\draw (2,6) node{$\bullet$};

\draw (-.5,7) node{$\times$};
\draw (.5,7) node{$\times$};
\draw (1.5,7) node{$\times$};

\draw (.5,.25) node[scale=.7]{0};
\draw (1.5,.25) node[scale=.7]{3};

\draw (1,.75) node[scale=.7]{5};

\draw (.5,2.25) node[scale=.7]{1};
\draw (.5,2.75) node[scale=.7]{0};

\draw (.5,4.25) node[scale=.7]{2};
\draw (2,4.25) node[scale=.7]{5};

\draw (1,4.75) node[scale=.7]{4};
\draw (2.5,4.75) node[scale=.7]{3};

\draw (1,6.25) node[scale=.7]{4};

\draw (0,6.75) node[scale=.7]{2};
\draw (1,6.75) node[scale=.7]{1};
 
\draw (1,-1) node{$q$};
\draw[>=latex,->](4.5,2.75) -- (6,2.75) ;
\draw (5.25,3.4) node[scale=.6]{overcollapse $\mathcal{C}_4$};

\end{scope}

    \begin{scope}[shift={(5,0)},scale=.6]

\draw (0,0) rectangle (2,.5);
\draw (0,1.5) rectangle (1,3);
\draw (0,5.5) -- (.5,5.25) -- (1.5,5.25) -- (2,5.5) -- (3,5.5) -- (3,4) -- (1,4) -- (.5,4.25) -- (0,4) -- cycle;

\draw (0,0) node{$\bullet$};
\draw (1,0) node{$\bullet$};
\draw (2,0) node{$\bullet$};

\draw (0,.5) node{$\times$};
\draw (2,.5) node{$\times$};

\draw (0,1.5) node{$\times$};
\draw (1,1.5) node{$\times$};

\draw (0,3) node{$\bullet$};
\draw (1,3) node{$\bullet$};

\draw (0,4) node{$\times$};
\draw (1,4) node{$\times$};
\draw (3,4) node{$\times$};

\draw (0,5.5) node{$\bullet$};
\draw (2,5.5) node{$\bullet$};
\draw (3,5.5) node{$\bullet$};

\draw (.5,5.25) node{$\times$};
\draw (1.5,5.25) node{$\times$};

\draw (.5,4.25) node{$\bullet$};

\draw (.5,-.25) node[scale=.7]{0};
\draw (1.5,-.25) node[scale=.7]{3};

\draw (1,.75) node[scale=.7]{5};

\draw (.5,1.75) node[scale=.7]{1};
\draw (.5,2.75) node[scale=.7]{0};

\draw (2,4.25) node[scale=.7]{5};

\draw (2.5,5.25) node[scale=.7]{3};
\draw (1,5.5) node[scale=.7]{1};
 
\draw (1,-1) node{$q'$};

\end{scope}
    \end{tikzpicture}
\end{center}

The surface $q'$ has three cylinders $\mathcal{C}'_1$, $\mathcal{C}'_2$ and $\mathcal{C}'_3$ coming from the cylinders $\mathcal{C}_1$, $\mathcal{C}_2$ and $\mathcal{C}_3$ on $q$ plus an additional cylinder $\mathcal{C}$ of height $\varepsilon$ that was created by the overcollapse of $\mathcal{C}_4$. Notice that for any $i \in \{1,2,3\}$, the circumference of the cylinder $\mathcal{C}'_i$ is the same as the one of $\mathcal{C}_i$, this is because our overcollapse deformation is isoperiodic and the circumference of a cylinder is the period of its core curve. Suppose to a contradiction that the circumferences of the cylinders $\mathcal{C}_1$ and $\mathcal{C}_2$ are not commensurable. It follows from Proposition \ref{twist} that there are two rational numbers $r,s$ such that $c^{-1}_3 = rc^{-1}_1 + sc^{-1}_2$. Denote by $\mu_1,\mu_2$ and $\mu_3$ (resp. $\mu'_1,\mu'_2$ and $\mu'_3$) the moduli of the cylinders $\mathcal{C}_1, \mathcal{C}_2$ and $\mathcal{C}_3$ of $q$ (resp. $\mathcal{C}'_1, \mathcal{C}'_2$ and $\mathcal{C}'_3$ of $q'$). The cylinder decompositions on $q$ and $q'$ are made of relative cylinders whose circumferences are not all pairwise commensurable since $\M$ is nonarithmetic by Proposition \ref{field}. By Proposition \ref{equation}, we know that:

\begin{align*}
    \mu_3 &= -r\mu_1 + s\mu_2 \\
    -\mu'_3 &= -r\mu'_1 + s\mu'_2 
\end{align*}

The important thing to notice here is the minus sign in front of $\mu_3'$ that does not appear in front of $\mu_3$. This is due to that fact that when $\mathcal{C}_1$ overcollapses, the type of $\mathcal{C}_3$ is changed. We also compute $\mu_3' = \mu_3 + c^{-1}_3(h_4 - \varepsilon)$, $\mu_1' = \mu_1 - c^{-1}_1(h_4+\varepsilon)$ and $\mu_2' = \mu_2 + c^{-1}_2(h_4+\varepsilon)$. We thus get: 

\begin{align*}
\mu_3 - \mu'_3 &= -r(\mu_1 + \mu_1') + s(\mu_2 + \mu_2') \\
    &= -2r\mu_1 + 2s\mu_2 + (h_4 + \varepsilon)(rc^{-1}_1 + sc^{-1}_2) \\
    &= 2\mu_3 + c^{-1}_3(h_4 + \varepsilon) \\
\end{align*}

This simplifies to $h_3 + h_4 = 0$, which is a contradiction. We deduce from this that the circumferences of $\mathcal{C}_1$ and $\mathcal{C}_2$ are commensurable and we get another contradiction by Proposition \ref{noneq} as these cylinders do not have same type. It follows that $h_1 \leq h_4$ and using the exact same argument switching the role of $\mathcal{C}_1$ and $\mathcal{C}_4$, we obtain that the cylinders $\mathcal{C}_1$ and $\mathcal{C}_4$ have same height. We conclude as in the previous case using the rank one assumption and the fact that $\mathcal{C}_1$ and $\mathcal{C}_4$ have same type that they are $\M$-entangled.

\textbf{Case 4.4} We claim that $\mathcal{C}_1$ and $\mathcal{C}_2$ are $\M$-entangled. To see this, shear the surface $q$ so that the cylinder $\mathcal{C}_4$ is not crossed by a vertical saddle connection and overcollapse the cylinder $\mathcal{C}_4$ using imaginary rel.  We reach a surface $q'$ whose cylinder decomposition is given by $4.3$, as depicted below: 

\begin{center}
\begin{tikzpicture}
\begin{scope}[scale=.6]

\draw (0,0) rectangle (2,1);
\draw (0,2) rectangle (2,3);
\draw (0,4) rectangle (1,5);
\draw (0,6) -- (-.5,7) -- (.5,7) -- (1,6) -- cycle;

\draw (0,0) node{$\bullet$};
\draw (1,0) node{$\bullet$};
\draw (2,0) node{$\bullet$};

\draw (0,1) node{$\bullet$};
\draw (1,1) node{$\bullet$};
\draw (2,1) node{$\bullet$};

\draw (0,2) node{$\times$};
\draw (1,2) node{$\times$};
\draw (2,2) node{$\times$};

\draw (0,3) node{$\times$};
\draw (1,3) node{$\times$};
\draw (2,3) node{$\times$};

\draw (0,4) node{$\times$};
\draw (1,4) node{$\times$};

\draw (0,5) node{$\bullet$};
\draw (1,5) node{$\bullet$};

\draw (0,6) node{$\bullet$};
\draw (1,6) node{$\bullet$};

\draw (-.5,7) node{$\times$};
\draw (.5,7) node{$\times$};

\draw (.5,.25) node[scale=.7]{0};
\draw (1.5,.25) node[scale=.7]{3};

\draw (.5,.75) node[scale=.7]{5};
\draw (1.5,.75) node[scale=.7]{0};

\draw (.5,2.25) node[scale=.7]{1};
\draw (1.5,2.25) node[scale=.7]{2};

\draw (.5,2.75) node[scale=.7]{1};
\draw (1.5,2.75) node[scale=.7]{4};

\draw (.5,4.25) node[scale=.7]{4};
\draw (.5,4.75) node[scale=.7]{3};

\draw (.5,6.25) node[scale=.7]{5};
\draw (0,6.75) node[scale=.7]{2};

\draw (1,-1) node{$q$};
\draw[>=latex,->](4.5,2.75) -- (6,2.75) ;
\draw (5.25,3.4) node[scale=.6]{overcollapse $\mathcal{C}_4$};

\end{scope}
\begin{scope}[shift={(5,0)},scale=.6]

\draw (0,0) -- (0,1) -- (.5,.75) -- (1,1) -- (2,1) -- (2,0) -- cycle;
\draw (0,2) -- (0,3) -- (2,3) -- (2,2) -- (1.5,2.25) -- (1,2) --  cycle;
\draw (0,4) rectangle (1,6);

\draw (0,0) node{$\bullet$};
\draw (1,0) node{$\bullet$};
\draw (2,0) node{$\bullet$};
\draw (.5,0.75) node{$\times$};

\draw (0,1) node{$\bullet$};
\draw (1,1) node{$\bullet$};
\draw (2,1) node{$\bullet$};
\draw (1.5,2.25) node{$\bullet$};

\draw (0,2) node{$\times$};
\draw (1,2) node{$\times$};
\draw (2,2) node{$\times$};

\draw (0,3) node{$\times$};
\draw (1,3) node{$\times$};
\draw (2,3) node{$\times$};

\draw (0,4) node{$\times$};
\draw (1,4) node{$\times$};

\draw (0,6) node{$\bullet$};
\draw (1,6) node{$\bullet$};

\draw (.5,.25) node[scale=.7]{0};
\draw (1.5,.25) node[scale=.7]{3};

\draw (1.5,.75) node[scale=.7]{0};

\draw (.5,2.25) node[scale=.7]{1};

\draw (.5,2.75) node[scale=.7]{1};
\draw (1.5,2.75) node[scale=.7]{4};

\draw (.5,4.25) node[scale=.7]{4};
\draw (.5,5.75) node[scale=.7]{3};

\draw (1,-1) node{$q'$};

\end{scope}
\end{tikzpicture}
\end{center}

One way to see that the cylinder diagram of $q'$ is given by $4.3$ is to notice that there are four relative cylinders, two of each type, which characterizes uniquely diagram 4.3. The surface $q'$ has three cylinders $\mathcal{C}'_1$, $\mathcal{C}'_2$ and $\mathcal{C}'_3$ coming from the cylinders $\mathcal{C}_1$, $\mathcal{C}_2$ and $\mathcal{C}_3$ on $q$ plus an additional cylinder $\mathcal{C}$ of height $\varepsilon$ that was created by the overcollapse of $\mathcal{C}_4$. Using our analysis of Case 4.3, we deduce that $\mathcal{C}'_1$ and $\mathcal{C}'_2$ have same height. But $h_1 = h_1' -\varepsilon = h_2' - \varepsilon = h_2$ and thus $\mathcal{C}_1$ and $\mathcal{C}_2$ have same height. Since they are absolute, they keep the same height on any nearby surface and thus they are $\M$-entangled.

\textbf{Case 4.1} We claim that this case cannot occur. To see this, shear the surface $q'$ so that $\mathcal{C}_4$ is not crossed by a vertical saddle connection and that the horizontal saddle connection labeled $2$ sits above the singularity $\bullet$ in $\mathcal{C}_4$. Overcollapsing $\mathcal{C}_4$ produces a horizontally periodic surface $q'$ whose cylinder diagram is given by 4.3, as depicted below:

\begin{center}
    \begin{tikzpicture}
    \begin{scope}[scale=.6]

\draw (0,0) rectangle (3,1);
\draw (0,2) rectangle (1,3);
\draw (0,4) rectangle (2,5);
\draw (0,6) -- (.5,7) -- (2.5,7) -- (2,6) --cycle;

\draw (0,0) node{$\bullet$};
\draw (1,0) node{$\bullet$};
\draw (3,0) node{$\bullet$};

\draw (0,1) node{$\bullet$};
\draw (2,1) node{$\bullet$};
\draw (3,1) node{$\bullet$};

\draw (0,2) node{$\times$};
\draw (1,2) node{$\times$};

\draw (0,3) node{$\times$};
\draw (1,3) node{$\times$};

\draw (0,4) node{$\times$};
\draw (1,4) node{$\times$};
\draw (2,4) node{$\times$};

\draw (0,5) node{$\bullet$};
\draw (2,5) node{$\bullet$};

\draw (0,6) node{$\bullet$};
\draw (2,6) node{$\bullet$};

\draw (0.5,7) node{$\times$};
\draw (1.5,7) node{$\times$};
\draw (2.5,7) node{$\times$};

\draw (.5,.25) node[scale=.7]{0};
\draw (2,.25) node[scale=.7]{1};

\draw (1,.75) node[scale=.7]{5};
\draw (2.5,.75) node[scale=.7]{0};

\draw (.5,2.25) node[scale=.7]{2};
\draw (.5,2.75) node[scale=.7]{4};

\draw (.5,4.25) node[scale=.7]{3};
\draw (1.5,4.25) node[scale=.7]{4};

\draw (1,4.75) node[scale=.7]{1};

\draw (1,6.25) node[scale=.7]{5};

\draw (1,6.75) node[scale=.7]{3};
\draw (2,6.75) node[scale=.7]{2};

\draw (2,-1) node{$q$};
\draw[>=latex,->](4.5,2.75) -- (6,2.75) ;
\draw (5.25,3.4) node[scale=.6]{overcollapse $\mathcal{C}_4$};

\end{scope}

    \begin{scope}[shift={(5,0)},scale=.6]

\draw (0,0) -- (0,1) -- (.5,.75) -- (1.5,.75) -- (2,1) -- (3,1) -- (3,0) -- cycle;
\draw (0,2) -- (0,3) -- (1,3) -- (1,2) -- (.5,2.25) -- cycle;
\draw (0,4) rectangle (2,6);

\draw (0,0) node{$\bullet$};
\draw (1,0) node{$\bullet$};
\draw (3,0) node{$\bullet$};

\draw (0,1) node{$\bullet$};
\draw (2,1) node{$\bullet$};
\draw (3,1) node{$\bullet$};
\draw (.5,.75) node{$\times$};
\draw (1.5,.75) node{$\times$};

\draw (0,2) node{$\times$};
\draw (1,2) node{$\times$};
\draw (.5,2.25) node{$\bullet$};

\draw (0,3) node{$\times$};
\draw (1,3) node{$\times$};

\draw (0,4) node{$\times$};
\draw (1,4) node{$\times$};
\draw (2,4) node{$\times$};

\draw (0,6) node{$\bullet$};
\draw (2,6) node{$\bullet$};

\draw (.5,.25) node[scale=.7]{0};

\draw (2,.25) node[scale=.7]{1};

\draw (2.5,.75) node[scale=.7]{0};

\draw (.5,2.75) node[scale=.7]{4};

\draw (.5,4.25) node[scale=.7]{3};
\draw (1.5,4.25) node[scale=.7]{4};

\draw (1,5.75) node[scale=.7]{1};

\draw (2,-1) node{$q'$};

\end{scope}
    \end{tikzpicture}
\end{center}

The surface $q'$ has three cylinders $\mathcal{C}'_1$, $\mathcal{C}'_2$ and $\mathcal{C}'_3$, coming from the cylinders $\mathcal{C}_1$, $\mathcal{C}_2$ and $\mathcal{C}_3$ on $q$ plus an additional cylinder $\mathcal{C}$ of height $\varepsilon$ that was created by the overcollapse of $\mathcal{C}_1$. Using our analysis of Case 4.3 we know that $\mathcal{C}$ and $\mathcal{C}_3'$ should be $\M$-entangled but the height of $\mathcal{C}$ is $\varepsilon$ while the one of $\mathcal{C}'_3$ is $h_3 + h_4 + \varepsilon$. This is a contradiction which rules out Case 4.1. 

\textbf{Case 4.5} We claim that $\mathcal{C}_3$ and $\mathcal{C}_4$ are $\M$-entangled. To see this, we first prove that $\mathcal{C}_3$ and $\mathcal{C}_4$ must have same circumference. Suppose to a contradiction that $c_3>c_4$ and shear the surface $q$ so that the midpoint of the saddle connection labeled $3$ in $\mathcal{C}_1$ sits above the midpoint of the saddle connection labeled $0$ and overcollapse $\mathcal{C}_1$ to reach a horizontally periodic surface $q'$ as follows: 

\begin{center}
    \begin{tikzpicture}
    \begin{scope}[scale=.6]

\draw (0.5,0) -- (0,1) -- (2,1) -- (2.5,0) -- cycle;
\draw (0,2) rectangle (2,3);
\draw (0,4) rectangle (1.5,5);
\draw (0,6) rectangle (.5,7);

\draw (0.5,0) node{$\bullet$};
\draw (1,0) node{$\bullet$};
\draw (2.5,0) node{$\bullet$};

\draw (0,1) node{$\times$};
\draw (1.5,1) node{$\times$};
\draw (2,1) node{$\times$};

\draw (0,2) node{$\times$};
\draw (.5,2) node{$\times$};
\draw (2,2) node{$\times$};

\draw (0,3) node{$\bullet$};
\draw (.5,3) node{$\bullet$};
\draw (2,3) node{$\bullet$};

\draw (0,4) node{$\bullet$};
\draw (1.5,4) node{$\bullet$};

\draw (0,5) node{$\bullet$};
\draw (1.5,5) node{$\bullet$};

\draw (0,6) node{$\times$};
\draw (.5,6) node{$\times$};

\draw (0,7) node{$\times$};
\draw (.5,7) node{$\times$};

\draw (.75,.25) node[scale=.7]{5};
\draw (1.75,.25) node[scale=.7]{0};

\draw (.75,.75) node[scale=.7]{4};
\draw (1.75,.75) node[scale=.7]{3};

\draw (.25,2.25) node[scale=.7]{1};
\draw (1.25,2.25) node[scale=.7]{4};

\draw (.25,2.75) node[scale=.7]{5};
\draw (1.25,2.75) node[scale=.7]{2};

\draw (.75,4.25) node[scale=.7]{2};
\draw (.75,4.75) node[scale=.7]{0};

\draw (.25,6.25) node[scale=.7]{3};
\draw (.25,6.75) node[scale=.7]{1};

\draw (1,-1) node{$q$};
\draw[>=latex,->](4.5,2.75) -- (6,2.75) ;
\draw (5.25,3.4) node[scale=.6]{overcollapse $\mathcal{C}_1$};

\end{scope}

    \begin{scope}[shift={(5,0)},scale=.6]
\draw (0,1) -- (.5,1) -- (1,1.25) -- (1.5,1.25) -- (2,1) -- (2,3) -- (0,3) --cycle;
\draw (0,4) -- (1.5,4) -- (1.5,5) -- (1,4.75) -- (.5,4.75) -- (0,5) -- cycle;
\draw (0,6) rectangle (.5,7);

\draw (0,1) node{$\times$};
\draw (.5,1) node{$\times$};
\draw (2,1) node{$\times$};

\draw (1,1.25) node{$\bullet$};
\draw (1.5,1.25) node{$\bullet$};

\draw (0,3) node{$\bullet$};
\draw (.5,3) node{$\bullet$};
\draw (2,3) node{$\bullet$};

\draw (0,4) node{$\bullet$};
\draw (1.5,4) node{$\bullet$};

\draw (0,5) node{$\bullet$};
\draw (1.5,5) node{$\bullet$};

\draw (.5,4.75) node{$\times$};
\draw (1,4.75) node{$\times$};

\draw (0,6) node{$\times$};
\draw (.5,6) node{$\times$};

\draw (0,7) node{$\times$};
\draw (.5,7) node{$\times$};

\draw (.25,1.25) node[scale=.7]{1};
\draw (1.25,2.25) node[scale=.7]{4};

\draw (.25,2.75) node[scale=.7]{5};
\draw (1.25,2.75) node[scale=.7]{2};

\draw (.75,4.25) node[scale=.7]{2};

\draw (.25,6.25) node[scale=.7]{3};
\draw (.25,6.75) node[scale=.7]{1};

\draw (1,-1) node{$q'$};
\end{scope}

    \end{tikzpicture}
\end{center}

It is easily seen that $q'$ has a horizontal cylinder decomposition as in Case 4.1, which is in contradiction with what was proved earlier and thus $c_3 \leq c_4$. We prove likewise that $c_4\leq c_3$ and this proves our first claim that $c_3 = c_4$. Now, shear the surface $q$ so that $\mathcal{C}_1$ is not crossed by a vertical saddle connection and that in $\mathcal{C}_1$, the midpoint of the horizontal saddle connection labeled 0 sits below the endpoint of the saddle connection labeled 4 (which is also the beginning of the saddle connection labeled 3). Overcollapsing cylinder $\mathcal{C}_1$ using imaginary rel produces a surface $q''$ whose cylinder diagram is given by $4.2$, as depicted below: 

\begin{center}
    \begin{tikzpicture}
    \begin{scope}[scale=.6]

\draw (0.5,0) -- (0,1) -- (2,1) -- (2.5,0) -- cycle;
\draw (0,2) rectangle (2,3);
\draw (0,4) rectangle (1,5);
\draw (0,6) rectangle (1,7);

\draw (0.5,0) node{$\bullet$};
\draw (1.5,0) node{$\bullet$};
\draw (2.5,0) node{$\bullet$};

\draw (0,1) node{$\times$};
\draw (1,1) node{$\times$};
\draw (2,1) node{$\times$};

\draw (0,2) node{$\times$};
\draw (1,2) node{$\times$};
\draw (2,2) node{$\times$};

\draw (0,3) node{$\bullet$};
\draw (1,3) node{$\bullet$};
\draw (2,3) node{$\bullet$};

\draw (0,4) node{$\bullet$};
\draw (1,4) node{$\bullet$};

\draw (0,5) node{$\bullet$};
\draw (1,5) node{$\bullet$};

\draw (0,6) node{$\times$};
\draw (1,6) node{$\times$};

\draw (0,7) node{$\times$};
\draw (1,7) node{$\times$};

\draw (1,.25) node[scale=.7]{0};
\draw (2,.25) node[scale=.7]{5};

\draw (.5,.75) node[scale=.7]{4};
\draw (1.5,.75) node[scale=.7]{3};

\draw (.5,2.25) node[scale=.7]{1};
\draw (1.5,2.25) node[scale=.7]{4};

\draw (.5,2.75) node[scale=.7]{5};
\draw (1.5,2.75) node[scale=.7]{2};

\draw (.5,4.25) node[scale=.7]{2};
\draw (.5,4.75) node[scale=.7]{0};

\draw (.5,6.25) node[scale=.7]{3};
\draw (.5,6.75) node[scale=.7]{1};

\draw (1,-1) node{$q$};
\draw[>=latex,->](4.5,2.75) -- (6,2.75) ;
\draw (5.25,3.4) node[scale=.6]{overcollapse $\mathcal{C}_1$};

\end{scope}

    \begin{scope}[shift={(5,0)},scale=.6]

\draw (0,1) -- (0,3) -- (.5,2.75) -- (1,3) -- (2,3) -- (2,1) -- (1.5,1.25) -- (1,1) -- cycle;
\draw (0,4) -- (0,5) -- (.5,4.75) -- (1,5) -- (1,4) -- cycle;
\draw (0,6) -- (0,7) -- (1,7) -- (1,6) -- (.5,6.25) -- cycle;

\draw (0,1) node{$\times$};
\draw (1,1) node{$\times$};
\draw (2,1) node{$\times$};
\draw (1.5,1.25) node{$\bullet$};

\draw (0,3) node{$\bullet$};
\draw (1,3) node{$\bullet$};
\draw (2,3) node{$\bullet$};
\draw (.5,2.75) node{$\times$};

\draw (0,4) node{$\bullet$};
\draw (1,4) node{$\bullet$};
\draw (.5,4.75) node{$\times$};

\draw (0,5) node{$\bullet$};
\draw (1,5) node{$\bullet$};

\draw (0,6) node{$\times$};
\draw (1,6) node{$\times$};
\draw (.5,6.25) node{$\bullet$};
\draw (0,7) node{$\times$};
\draw (1,7) node{$\times$};

\draw (.5,1.25) node[scale=.7]{1};

\draw (1.5,2.75) node[scale=.7]{2};

\draw (.5,4.25) node[scale=.7]{2};

\draw (.5,6.75) node[scale=.7]{1};

\draw (1,-1) node{$q''$};
\end{scope}

    \end{tikzpicture}
\end{center}

The surface $q''$ has three cylinders $\mathcal{C}'_2$, $\mathcal{C}'_3$ and $\mathcal{C}'_4$ coming from the cylinders $\mathcal{C}_2$, $\mathcal{C}_3$ and $\mathcal{C}_4$ on $q$ plus an additional cylinder $\mathcal{C}$ of height $\varepsilon$ that was created by the overcollapse of $\mathcal{C}_1$. Using our analysis of Case 4.2, we know that two of the cylinders $\mathcal{C}'_2$, $\mathcal{C}'_3$ and $\mathcal{C}'_4$ are $\M$-entangled. We compute that the height of $\mathcal{C}'_2$ is $h_2+h_1-\varepsilon$, the one of $\mathcal{C}'_3$ is $h_3 - \varepsilon$ and the one of $\mathcal{C}'_4$ is $h_4 - \varepsilon$. Suppose to a contradiction that $\mathcal{C}'_2$ is $\M$-entangled to $\mathcal{C}'_3$. It follows from Corollary \ref{height} that their circumferences, which are $c_2$ and $c_3$, are commensurable. The core curves of $\mathcal{C}_1$ and $\mathcal{C}_2$ are homologous so $c_1 = c_2$ and since $c_3 = c_4$ as proved earlier, we deduce that the circumferences of the horizontal cylinders on $q$ are pairwise commensurable, which is a contradiction with the fact that $\M$ is nonarithmetic by Proposition \ref{field}. We prove likewise that $\mathcal{C}'_2$ is not $\M$-entangled to $\mathcal{C}'_4$ and thus we have proved that $\mathcal{C}'_3$ and $\mathcal{C}'_4$ are $\M$-entangled. But $h_3 = h_3'+\varepsilon = h_4'+\varepsilon = h_4$ and thus $\mathcal{C}_3$ and $\mathcal{C}_4$ have same height. Since they are absolute, they keep the same height on any nearby surface and thus they are $\M$-entangled.

\end{proof}

\textbf{Remark.} Lemma 6.17 in \cite{aulicino2016rank} states that Case 4.1 does not occur as the cylinder diagram of a surface in a rank 2 affine invariant orbifold either. This should be reassuring in the perspective of establishing Theorem \ref{T1} since any connected component of $\Omega E_D^{odd}(2,2)$ lives in a rank two affine invariant orbifold, namely the Prym locus.

\begin{prop}\label{reduction}
Let $\M$ be a rel-invariant nonarithmetic affine invariant orbifold in $\mathcal{H}^{odd}(2,2)$. It contains a horizontally periodic surface whose cylinder diagram is given by 4.2 or 4.3 in the notation of Appendix A. 
\end{prop}

\bigskip 

\begin{proof}

\begin{center}
\begin{tikzpicture}

\begin{scope}[scale=.6]

\draw (1,0) -- (4,0) -- (3,1) -- (0,1) -- cycle;
\draw (0,2) rectangle (3,3);
\draw (0,4) rectangle (1,5);
\draw[fill=black, opacity=.3] (1,0) rectangle (2,1);
\draw[fill=black, opacity=.3] (2,2) rectangle (3,3);

\draw (1,0) node{$\bullet$};
\draw (2,0) node{$\bullet$};
\draw (3,0) node{$\bullet$};
\draw (4,0) node{$\bullet$};

\draw (0,1) node{$\times$}; 
\draw (2,1) node{$\times$}; 
\draw (3,1) node{$\times$}; 

\draw (0,2) node{$\times$}; 
\draw (1,2) node{$\times$}; 
\draw (3,2) node{$\times$}; 

\draw (0,3) node{$\bullet$};
\draw (1,3) node{$\bullet$};
\draw (2,3) node{$\bullet$};
\draw (3,3) node{$\bullet$};

\draw (0,4) node{$\times$};
\draw (1,4) node{$\times$};

\draw (0,5) node{$\times$};
\draw (1,5) node{$\times$};

\draw (1.5,.25) node[scale=.7]{0};
\draw (2.5,.25) node[scale=.7]{2};
\draw (3.5,.25) node[scale=.7]{1};

\draw (1,.75) node[scale=.7]{5};
\draw (2.5,.75) node[scale=.7]{4};

\draw (.5,2.25) node[scale=.7]{3};
\draw (2,2.25) node[scale=.7]{5};

\draw (.5,2.75) node[scale=.7]{1};
\draw (1.5,2.75) node[scale=.7]{2};
\draw (2.5,2.75) node[scale=.7]{0};

\draw (.5,4.25) node[scale=.7]{4};
\draw (.5,4.75) node[scale=.7]{3};

\draw (1.5,-1) node{The vertical cylinder $\mathcal{C}$ depicted in grey};
\end{scope}    

\end{tikzpicture}
\end{center}

It follows from Proposition \ref{periodicity} that the surface $q$ is vertically periodic. Applying a well-chosen real rel deformation, small enough so that none of the vertical cylinders of $q$ collapses creates a new surface $q'$ that is still vertically periodic, that the corresponding vertical cylinders are all stable (applying real rel breaks the vertical saddle connection joining different singularities) and that the cylinder coming coming from $\mathcal{C}$ is now simple (which means that each of its boundary components is a single saddle connection) and relative, as depicted in the following picture: 

\begin{center}
\begin{tikzpicture}

\begin{scope}[scale=.6]

\draw (1.5,0) -- (4.5,0) -- (3.5,1) -- (0,1) -- cycle;
\draw (0,2) -- (3,2) -- (3.5,3) -- (.5,3) -- cycle;
\draw (0,4) rectangle (1,5);
\draw[fill=black, opacity=.3] (1.5,0) rectangle (2,1);
\draw[fill=black, opacity=.3] (2.5,2) rectangle (3,3);

\draw (1.5,0) node{$\bullet$};
\draw (2.5,0) node{$\bullet$};
\draw (3.5,0) node{$\bullet$};
\draw (4.5,0) node{$\bullet$};

\draw (0,1) node{$\times$}; 
\draw (2,1) node{$\times$}; 
\draw (3,1) node{$\times$}; 

\draw (0,2) node{$\times$}; 
\draw (1,2) node{$\times$}; 
\draw (3,2) node{$\times$}; 

\draw (0.5,3) node{$\bullet$};
\draw (1.5,3) node{$\bullet$};
\draw (2.5,3) node{$\bullet$};
\draw (3.5,3) node{$\bullet$};

\draw (0,4) node{$\times$};
\draw (1,4) node{$\times$};

\draw (0,5) node{$\times$};
\draw (1,5) node{$\times$};

\draw (1.5,.25) node[scale=.7]{0};
\draw (2.5,.25) node[scale=.7]{2};
\draw (3.5,.25) node[scale=.7]{1};

\draw (1,.75) node[scale=.7]{5};
\draw (2.5,.75) node[scale=.7]{4};

\draw (.5,2.25) node[scale=.7]{3};
\draw (2,2.25) node[scale=.7]{5};

\draw (.5,2.75) node[scale=.7]{1};
\draw (1.5,2.75) node[scale=.7]{2};
\draw (2.5,2.75) node[scale=.7]{0};

\draw (.5,4.25) node[scale=.7]{4};
\draw (.5,4.75) node[scale=.7]{3};

\draw (1.5,-1) node{A simple and relative vertical cylinder};

\end{scope}    

\end{tikzpicture}
\end{center}

Having only stable cylinders, one of which being simple and relative is a feature shared only by diagrams with four cylinders. Up to rotating by 90 degrees, we can assume that $q'$ has four horizontal cylinders.  Finally, we saw in the course of the proof of Proposition \ref{atmost} that the corresponding cylinder diagram cannot be given by Case 4.1 and that the Cases 4.4 and 4.5 can always be deformed inside $\M$ into a surface with a cylinder diagram given by Cases 4.3 or 4.2. 
\end{proof}

\begin{thm}[Theorem A]
Let $\mathcal{M}$ be a rel-invariant rank one affine invariant orbifold in $\mathcal{H}^{odd}(2,2)$. If $\M$ is nonarithmetic, then there exists $D$ not a square such that $\M$ is a connected component of $\Omega E_D^{odd}(2,2)$.  
\end{thm}

\begin{proof}
Let $q \in \M$ be a horizontally periodic surface whose cylinder diagram is given by 4.2 or 4.3 in the notation of Appendix A. Such a surface is given by Proposition \ref{reduction} and we recall that we number the cylinders from bottom to top. We analyse the two cases separately: 

\textbf{Case 4.2} By Proposition \ref{atmost}, we can assume without loss of generality that $\mathcal{C}_2$ and $\mathcal{C}_3$ are $\M$-entangled. We will show that they are actually isometric. Assume first that $c_2 \geq c_3$ and shear the surface using real rel and the horocycle flow so that $q$ has a vertical cylinder $\mathcal{C}$ that has the core curve of the horizontal cylinder $\mathcal{C}_2$ as a cross curve, as depicted in the following picture:

\begin{center}
    \begin{tikzpicture}
    \begin{scope}[scale = .6]

\draw (0,0) rectangle (3,1);
\draw (2,0) rectangle (3,1);
\draw[fill=black, opacity=.3] (2,0) rectangle (3,1);
\draw[fill=black, opacity=.3] (0,2) rectangle (1,3);
\draw (0,2) rectangle (1,3);
\draw (0,4) -- (.5,5) -- (1.5,5) -- (1,4) -- cycle;
\draw (0,6) -- (-.5,7) -- (.5,7) -- (1,6) -- cycle;

\draw (0,0) node{$\bullet$}; 
\draw (1,0) node{$\bullet$}; 
\draw (2,0) node{$\bullet$}; 
\draw (3,0) node{$\bullet$}; 

\draw (0,1) node{$\times$}; 
\draw (1,1) node{$\times$}; 
\draw (2,1) node{$\times$}; 
\draw (3,1) node{$\times$}; 

\draw (0,2) node{$\times$};
\draw (1,2) node{$\times$};

\draw (0,3) node{$\bullet$};
\draw (1,3) node{$\bullet$};

\draw (0,4) node{$\times$};
\draw (1,4) node{$\times$};

\draw (0.5,5) node{$\bullet$};
\draw (1.5,5) node{$\bullet$};

\draw (0,6) node{$\times$};
\draw (1,6) node{$\times$};

\draw (-.5,7) node{$\bullet$};
\draw (.5,7) node{$\bullet$};

\draw (.5,.25) node[scale=.7]{0};
\draw (1.5,.25) node[scale=.7]{2};
\draw (2.5,.25) node[scale=.7]{1};

\draw (.5,.75) node[scale=.7]{5};
\draw (1.5,.75) node[scale=.7]{4};
\draw (2.5,.75) node[scale=.7]{3};

\draw (.5,2.25) node[scale=.7]{3};
\draw (.5,2.75) node[scale=.7]{1};

\draw (.5,4.25) node[scale=.7]{4};
\draw (1,4.75) node[scale=.7]{2};

\draw (.5,6.25) node[scale=.7]{5};
\draw (0,6.75) node[scale=.7]{0};

\draw (1.5,-1) node{The vertical cylinder $\mathcal{C}$ in grey};

\end{scope}
    \end{tikzpicture}
\end{center}

It follows from Proposition \ref{periodicity} that the vertical direction is periodic and thus the surface $q$ is decomposed into vertical cylinders. In particular, the length of any horizontal saddle connection is a linear combination of the heights of the vertical cylinders, the coefficients being given by the intersection numbers of the saddle connection with the core curves of the vertical cylinders. We claim that there are two positive numbers $l_1$ and $l_2$ such that the height of any of the vertical cylinders on $q$ is either $l_1$ or $l_2$. Suppose to a contradiction that there are at least three vertical cylinders on $q$ with pairwise distinct heights. Notice that we can apply real rel to $q$ to create a new vertical cylinder with arbitrarily small height adjacent to $\mathcal{C}$ without collapsing the other three cylinders (though their height might have changed). See Figure 4.1 in \cite{apisa2017rank} for a depiction of that phenomenon. In particular, this would create a surface in $\M$ with four cylinders with pairwise distinct height, in contradiction with Proposition \ref{atmost}. This proves our claim. Notice that $l_1$ and $l_2$ are not commensurable as otherwise all the circumferences of the horizontal cylinders, which are linear combinations of height of vertical cylinders, would be commensurable and $\M$ would be arithmetic by Proposition \ref{field}. Without lost of generality, we can assume that the height of $\mathcal{C}$ is $l_1$, so $l_1 = c_2$. We can find integers $m,n$ such that $c_3 = m l_1 + n l_2$. Notice that all the horizontal cylinders of $q$ are relative and so it follows from Proposition \ref{field} that we can apply Corollary \ref{height} to $\mathcal{C}_2$ and $\mathcal{C}_3$: since they have same type and same height, we know that $c_2$ and $c_3$ are commensurable. This is only possible if $n=0$ and because we assumed that $c_2 \geq c_3$ we also deduce that $m=1$. In particular $c_2 = c_3$ and the core curve of $\mathcal{C}_3$ is the cross curve a vertical cylinder of height $l_1 = c_2$. This proves that $\mathcal{C}_2$ and $\mathcal{C}_3$ are isometric. Let $\lambda$ be the involution that exchanges $\mathcal{C}_2$ and $\mathcal{C}_3$ while fixing $\mathcal{C}_1$ and $\mathcal{C}_4$ and such that $\lambda^{\ast}\omega_q = -\omega_q$. It has four fixed points located on the core curves of $\mathcal{C}_1$ and $\mathcal{C}_4$ and thus by Riemann-Hurwitz the surface $X_q / \lambda$ has genus 1, which proves that $\lambda$ is a Prym involution. That involution persists under small rel deformations and under the $\GL$-action, which means that any nearby surface is also endowed with a Prym involution. We can thus conclude from \cite[Theorem 7.3]{filip2016semisimplicity} that $\M$ is contained in $\Omega E^{odd}_D(2,2)$ with $D$ not a square and by a dimension count that $\M$ coincides with a connected component of $\Omega E_D(2,2)$. Alternatively, we can conclude in a more elementary way as follows: an application of the Closing Lemma for the geodesic flow that appears, for instance, in Lemma 4.1 of \cite{wright2014field} together with Birkhoff Ergodic Theorem applied to a compact set of large measure, shows that any neighborhood of $q$ contains a surface fixed by a hyperbolic matrix $A$. Up to replacing $q$ by this nearby surface, we conclude from Theorem 3.5 in \cite{mcmullen2007prym} that $q$ is a Prym eigenform in $\Omega E^{odd}_D(2,2)$ with $\mathbb{Q}(\sqrt{D})= \mathbb{Q}(tr(A))$. Notice that if $D$ is a square, then $tr(A)$ is rational and $q$ would cover a torus by Theorem 9.8 in \cite{mcmullen2003teichmuller} and thus $\M$ would be arithmetic. So $D$ is not a square and up to replacing $q$ by a nearby surface whose $\GL$-orbit closure is $\M$, we can conclude that $\M$ is a connected component of $\Omega E^{odd}_D(2,2)$. In the case where $c_2 \leq c_3$ we proceed exactly in the same way but this time shearing the surface $q$ so that the core curve of $\mathcal{C}_3$ is the cross curve of a vertical cylinder and we show again that $\mathcal{C}_2$ and $\mathcal{C}_3$ are isometric. 

\textbf{Case 4.3} By Proposition \ref{atmost}, we know that $\mathcal{C}_1$ and $\mathcal{C}_4$ are $\M$-entangled and we will show that they are actually isometric. Without loss of generality, we can assume that $c_4 \geq c_1$ and shear the surface $q$ using the real rel flow and the horocycle flow such that it has a vertical cylinder $\mathcal{C}$  with the horizontal saddle connection labeled 2 as a cross curve, as depicted in the following picture:

\begin{center}
    \begin{tikzpicture}
    \begin{scope}[scale=.6]

\draw (0,0) -- (.5,1) -- (2.5,1) -- (2,0) -- cycle;
\draw (0,2) -- (-.5,3) -- (.5,3) -- (1,2) -- cycle;
\draw[fill=black, opacity=.3] (0,4) rectangle (1,5);
\draw[fill=black, opacity=.3] (0,6) rectangle (1,7);
\draw (0,4) rectangle (3,5);
\draw (0,6) rectangle (2,7);

\draw (0,0) node{$\bullet$};
\draw (1,0) node{$\bullet$};
\draw (2,0) node{$\bullet$};

\draw (0.5,1) node{$\times$};
\draw (2.5,1) node{$\times$};

\draw (0,2) node{$\times$};
\draw (1,2) node{$\times$};

\draw (-.5,3) node{$\bullet$};
\draw (.5,3) node{$\bullet$};

\draw (0,4) node{$\times$};
\draw (1,4) node{$\times$};
\draw (3,4) node{$\times$};

\draw (0,5) node{$\bullet$};
\draw (2,5) node{$\bullet$};
\draw (3,5) node{$\bullet$};

\draw (0,6) node{$\bullet$};
\draw (2,6) node{$\bullet$};

\draw (0,7) node{$\times$};
\draw (1,7) node{$\times$};
\draw (2,7) node{$\times$};

\draw (.5,.25) node[scale=.7]{0};
\draw (1.5,.25) node[scale=.7]{3};

\draw (1.5,.75) node[scale=.7]{5};

\draw (.5,2.25) node[scale=.7]{1};
\draw (0,2.75) node[scale=.7]{0};

\draw (.5,4.25) node[scale=.7]{2};
\draw (2,4.25) node[scale=.7]{5};

\draw (1,4.75) node[scale=.7]{4};
\draw (2.5,4.75) node[scale=.7]{3};

\draw (1,6.25) node[scale=.7]{4};

\draw (.5,6.75) node[scale=.7]{2};
\draw (1.5,6.75) node[scale=.7]{1};
 
\draw (1,-1) node{The vertical cylinder $\mathcal{C}$ in grey};

\end{scope}
    \end{tikzpicture}
\end{center}

It follows from Proposition \ref{periodicity} that the vertical direction is periodic. We show as in the previous case that there are positive numbers $l_1$ and $l_2$ such that the height of any of the corresponding vertical cylinders is equal to either $l_1$ or $l_2$. Here again, it is not the case these two numbers are commensurable and we assume that the height of $\mathcal{C}$ is $l_1$, so the horizontal saddle connection labeled $2$ has length $l_1$. We can find integers $m,n$ such that the length of the horizontal saddle connection labeled 3 is $ml_1 + nl_2$. Since the saddle connections labeled $1$ and $0$ have same length and we assumed $c_4 \geq c_1$, we either have $m=1$ and $n=0$ or $m=0$. Assume to a contradiction that $m=0$ and let $a,b$ be two non-negative integers such that the saddle connection labeled $1$ has length $al_1 + bl_2$. Since $\mathcal{C}_1$ and $\mathcal{C}_4$ have same type and same height, it follows again from Corollary \ref{height} that there is $r \in \mathbb{Q}$ such that $c_4 = r c_1$. Using again the fact that the horizontal saddle connections labeled 1 and 0 have same length, this equation can be written: 

\begin{equation*}
    l_1 + al_1 + bl_2 = r(al_1 + bl_2 + nl_2)
\end{equation*}

Since $l_1$ and $l_2$ are not commensurable, we would get that $r = \frac{1+a}{a} = \frac{b}{b+n}$ which is impossible since $a,b$ and $n$ are non-negative integers. We thus have $m=1$ and $n=0$. This means that the saddle connection labeled $3$ is the cross curve of a vertical cylinder $\mathcal{C}'$ of height $l_1$. We claim that the cylinders $\mathcal{C}$ and $\mathcal{C}'$ are $\M$-entangled. Indeed, let $\mathcal{C}''$ be a vertical cylinder on $q$ of height $l_2$. Let $L = max(l_1,l_2)$, let $l = min(l_1,l_2)$, let $|\varepsilon| < \frac{1}{2}min(l_1,l_2,L-l)$ and let $q_{\varepsilon} = Rel^{\varepsilon}_v(q)$. The upper bound on $\varepsilon$ ensures that $q_{\varepsilon}$ has three vertical cylinders $\mathcal{C}_{\varepsilon}$, $\mathcal{C}'_{\varepsilon}$ and $\mathcal{C}''_{\varepsilon}$ of height strictly greater than $\varepsilon$ coming from $\mathcal{C}$, $\mathcal{C}'$ and $\mathcal{C}''$ plus an additional one of height $\varepsilon$. It also implies that the height of $\mathcal{C}_{\varepsilon}''$ is distinct from that of $\mathcal{C}_{\varepsilon}$ and $\mathcal{C}_{\varepsilon}'$. It follows from Proposition \ref{atmost} that $\mathcal{C}_{\varepsilon}$ and $\mathcal{C}_{\varepsilon}'$ have same height. Since $\varepsilon$ is arbitrary, this proves our claim that $\mathcal{C}$ and $\mathcal{C}'$ are $\M$-entangled. Let now $|\eta| < \frac{1}{2}min(l_1,l_2,L-l)$ such that on $q_{\eta}$ the cylinder $\mathcal{C}_{\eta}$ is stable and relative:

\begin{center}
    \begin{tikzpicture}
    \begin{scope}[scale=.6]

\draw (0,0) -- (.5,1) -- (2.5,1) -- (2,0) -- cycle;
\draw (0,2) -- (-.5,3) -- (.5,3) -- (1,2) -- cycle;
\draw[fill=black, opacity=.3] (0.5,4) rectangle (1,5);
\draw[fill=black, opacity=.3] (0.5,6) rectangle (1,7);
\draw (0,4) -- (.5,5) -- (3.5,5) -- (3,4) -- cycle;
\draw (0.5,6) -- (0,7) -- (2,7) -- (2.5,6) -- cycle;

\draw (0,0) node{$\bullet$};
\draw (1,0) node{$\bullet$};
\draw (2,0) node{$\bullet$};

\draw (0.5,1) node{$\times$};
\draw (2.5,1) node{$\times$};

\draw (0,2) node{$\times$};
\draw (1,2) node{$\times$};

\draw (-.5,3) node{$\bullet$};
\draw (.5,3) node{$\bullet$};

\draw (0,4) node{$\times$};
\draw (1,4) node{$\times$};
\draw (3,4) node{$\times$};

\draw (0.5,5) node{$\bullet$};
\draw (2.5,5) node{$\bullet$};
\draw (3.5,5) node{$\bullet$};

\draw (0.5,6) node{$\bullet$};
\draw (2.5,6) node{$\bullet$};

\draw (0,7) node{$\times$};
\draw (1,7) node{$\times$};
\draw (2,7) node{$\times$};

\draw (.5,.25) node[scale=.7]{0};
\draw (1.5,.25) node[scale=.7]{3};

\draw (1.5,.75) node[scale=.7]{5};

\draw (.5,2.25) node[scale=.7]{1};
\draw (0,2.75) node[scale=.7]{0};

\draw (.5,4.25) node[scale=.7]{2};
\draw (2,4.25) node[scale=.7]{5};

\draw (1.5,4.75) node[scale=.7]{4};
\draw (3,4.75) node[scale=.7]{3};

\draw (1.5,6.25) node[scale=.7]{4};

\draw (.5,6.75) node[scale=.7]{2};
\draw (1.5,6.75) node[scale=.7]{1};
 
\draw (1,-1) node{$\mathcal{C}_{\eta}$ is relative on $q_{\eta}$};
\end{scope}
    \end{tikzpicture}
\end{center}

The cylinder $\mathcal{C}_{\eta}'$ is also relative with same type as $\mathcal{C}_{\eta}$ on $q_{\eta}$ as otherwise the two cylinders could not keep same height when $\eta$ varies. By Corollary \ref{height} this means that the circumferences of $\mathcal{C}_{\eta}$ and $\mathcal{C}'_{\eta}$ are commensurable and thus so are the circumferences of $\mathcal{C}$ and $\mathcal{C}'$. The circumference of $\mathcal{C}$ is $h_1 + h_3$ and notice that the circumference of $\mathcal{C}'$ is equal to $s(h_1 + h_2) + (s+1)(h_1 + h_3)$ where $s$ is the number of time the core curve of $\mathcal{C}'$ intersects the horizontal saddle connection labeled $1$. This is best seen with a carefully chosen basis of $H_1(X_q,\mathbb{Z})$: let $e_1$ be the homology class represented by the saddle connection labeled $1$, $e_2$ the one represented by the saddle connection $2$, $e_3$ the one represented by the saddle connection labeled $3$, $f_1$ the one represented by a cycle that crosses all the horizontal cylinder exactly once, $f_2$ the one represented by the core curve of $\mathcal{C}$ and finally $f_3$ the homology class represented by a cycle that intersects the saddle connection labeled $3$ and that crosses $\mathcal{C}_1$ and $\mathcal{C}_3$ exactly once. This basis $(e_i,f_i)_{i \in \{1,2,3\}}$ is symplectic with respect to the intersection product on $H_1(X_q,\mathbb{Z})$ and thus the cohomology class in $H^1(X_q,\mathbb{C})$ given by integration of the imaginary part of $\omega_q$ is equal to $\sum_{i=1}^3( \int_{f_i} \mathcal{I}m(\omega_q) ) e_i^{\ast} = (h_1+h_2+h_3+h_4)e_1^{\ast}$ + $(h_3+h_4) e_2^{\ast} + (h_1+h_3)e_3^{\ast}$ where $e_i^{\ast}$ denote the intersection product against $e_i$ and $h_1 = h_4$ since $\mathcal{C}_1$ and $\mathcal{C}_4$ are $\M$-entangled. Given that the core curve $\gamma'$ of $\mathcal{C}'$ does not intersect $e_2$ and intersects $e_3$ only once, we deduce our formula for the circumference of $\mathcal{C}'$, which is given by the integral of $\mathcal{I}m (\omega_q)$ over $\gamma'$. This analysis also shows that the circumference of any other vertical cylinders is a linear combination of $h_1+h_3$ and $h_1+h_2$ with integer coefficients. It follows once again from Proposition \ref{field} that these two numbers must be commensurable. It then comes from the fact the circumferences of the cylinders $\mathcal{C}$ and $\mathcal{C}'$ are commensurable that $s=0$. To sum up our analysis, we know that the vertical cylinder $\mathcal{C}'$ has the saddle connection labeled $3$ as a cross curve and that its circumference is $h_1+h_3$. The only possibility for this to happen is that $\mathcal{C}_1$ and $\mathcal{C}_4$ are isometric, as depicted below: 

\begin{center}
    \begin{tikzpicture}
    \begin{scope}[scale=.6]

\draw (0,0) rectangle (2,1);
\draw (0,2) -- (-.5,3) -- (.5,3) -- (1,2) -- cycle;
\draw (0,4) rectangle (3,5);
\draw (0,6) rectangle (2,7);
\draw[fill=black, opacity=.6] (1,0) rectangle (2,1);
\draw[fill=black, opacity=.6] (2,4) rectangle (3,5);
\draw[fill=black, opacity=.3] (0,6) rectangle (1,7);
\draw[fill=black, opacity=.3] (0,4) rectangle (1,5);

\draw (0,0) node{$\bullet$};
\draw (1,0) node{$\bullet$};
\draw (2,0) node{$\bullet$};

\draw (0,1) node{$\times$};
\draw (2,1) node{$\times$};

\draw (0,2) node{$\times$};
\draw (1,2) node{$\times$};

\draw (-.5,3) node{$\bullet$};
\draw (.5,3) node{$\bullet$};

\draw (0,4) node{$\times$};
\draw (1,4) node{$\times$};
\draw (3,4) node{$\times$};

\draw (0,5) node{$\bullet$};
\draw (2,5) node{$\bullet$};
\draw (3,5) node{$\bullet$};

\draw (0,6) node{$\bullet$};
\draw (2,6) node{$\bullet$};

\draw (0,7) node{$\times$};
\draw (1,7) node{$\times$};
\draw (2,7) node{$\times$};

\draw (.5,.25) node[scale=.7]{0};
\draw (1.5,.25) node[scale=.7]{3};

\draw (1,.75) node[scale=.7]{5};

\draw (.5,2.25) node[scale=.7]{1};
\draw (0,2.75) node[scale=.7]{0};

\draw (.5,4.25) node[scale=.7]{2};
\draw (2,4.25) node[scale=.7]{5};

\draw (1,4.75) node[scale=.7]{4};
\draw (2.5,4.75) node[scale=.7]{3};

\draw (1,6.25) node[scale=.7]{4};

\draw (.5,6.75) node[scale=.7]{2};
\draw (1.5,6.75) node[scale=.7]{1};
 
\draw (1,-1) node{The vertical cylinders $\mathcal{C}$ and $\mathcal{C}'$ are isometric};

\end{scope}
    \end{tikzpicture}
\end{center}

From there, we can conclude in two different ways. The first one is as follows: let $\lambda$ be the involution that fixes $\mathcal{C}_2$ and $\mathcal{C}_3$ while exchanging $\mathcal{C}_1$ and $\mathcal{C}_4$ and such that $\lambda^{\ast}\omega_q = -\omega_q$. It has four fixed points located on the core curves of $\mathcal{C}_2$ and $\mathcal{C}_3$ and thus by Riemann-Hurwitz the surface $X_q / \lambda$ has genus 1, which proves that $\lambda$ is a Prym involution. We conclude as in the previous case that $\M$ is a connected component of some $\Omega E^{odd}_D(2,2)$ with $D$ not a square. The other way to conclude is as follows: shear the surface $q$ so that the saddle connection labeled $2$ in $\mathcal{C}_4$ sits above the singularity $\bullet$ and that $\mathcal{C}_4$ is not crossed by a vertical saddle connection. Since $\mathcal{C}_1$ and $\mathcal{C}_4$ are isometric, it is also the case that the saddle connection labeled $3$ in $\mathcal{C}_1$ sits bellow the singularity $\times$ and that $\mathcal{C}_1$ is not crossed by a vertical saddle connection. We can thus overcollapse simultaneously $\mathcal{C}_1$ and $\mathcal{C}_4$. Notice that the surface we reach is horizontally periodic with saddle connection diagram given by $4.2$ and we can use our analysis of Case $4.2$ to conclude the proof of Theorem \ref{T1}.
\end{proof}

\textbf{Remark.} Along the proof of Theorem \ref{T1}, we showed in particular that if $\M$ is a nonarithmetic rel-invariant rank one affine invariant orbifold, then it contains a horizontally periodic surface corresponding to the cylinder diagram $4.2$. This is coherent with Theorem B in \cite{lanneau2014connected} in the case $\kappa = (2,2)^{odd}$.

\begin{cor}\label{veech22}
If $q$ is a nonarithmetic Veech surface in $\mathcal{H}^{odd}(2,2)$ that is not contained in the Prym locus, then the subset $\GL \cdot\mathcal{F}_{q}$ is dense in $\mathcal{H}^{odd}(2,2)$.  
\end{cor}

\begin{proof}
Let $\M$ be the closure of $\GL \cdot\mathcal{F}_{q}$. This is an affine invariant orbifold by \cite{eskin2015isolation} as it is connected, closed, and $\GL$ invariant, and it is rel-invariant. If the rank of $\M$ is 1 then by Lemma 2.11 in \cite{lanneau2017finiteness} we have that $\M$ is nonarithmetic and it would follow from Theorem \ref{T1} that $q$ is contained in the Prym locus which we assumed is not the case. The rank of $\M$ is thus at least 2 but by Theorem 1.1 in \cite{aulicino2016rank} we know that the only rel-invariant rank 2 affine invariant orbifold is the Prym locus. The rank of $\M$ is thus 3 and by Theorem 1.1 in \cite{mirzakhani2018full} this is actually the whole stratum $\mathcal{H}^{odd}(2,2)$ as the hyperelliptic locus is not rel-invariant either.
\end{proof}

\section{The stratum $\mathcal{H}(3,1)$}\label{H31}

\begin{prop}\label{combinatorix31}
Let $\M \subset \mathcal{H}(3,1)$ be a rel-invariant affine invariant orbifold. There is a horizontally periodic surface $q \in \M$ made of four stable cylinders $\mathcal{C}_1, \cdots, \mathcal{C}_4$ with $\mathcal{C}_1$ in $\mathfrak{C}^-$, $\mathcal{C}_2$ and $\mathcal{C}_3$ in $\mathfrak{C}^+$, $\mathcal{C}_4$ in $\mathfrak{C}^0$ and such that there is a saddle connection $\sigma$ that appears both on the bottom boundary component of $\mathcal{C}_1$ and the top boundary component of $\mathcal{C}_4$, as depicted below:

\medskip 

\begin{center}
\begin{tikzpicture}[scale=.5]
\draw[dashed] (0,3) -- (3,3);
\draw (3,3) -- (4,2) -- (5,3);
\draw[dashed] (5,3) -- (8,3);
\draw (8,3) -- (7,2) -- (7,1);
\draw[dashed] (1,1) -- (3.5,1);
\draw[dashed] (4.5,1) -- (7,1);
\draw (3.5,1) -- (4.5,1);
\draw (1,1) -- (1,2) -- (0,3);
\draw[dashed] (10,2.5) -- (12,2.5);
\draw[dashed] (13,2.5) -- (15,2.5);
\draw[dashed] (10,1.5) -- (15,1.5);
\draw (10,1.5) -- (10,2.5);
\draw (15,1.5) -- (15,2.5);
\draw (12,2.5) -- (13,2.5);

\draw (2,2.4) node[scale=.7]{$\mathcal{C}_2$};
\draw (6,2.4) node[scale=.7]{$\mathcal{C}_3$};
\draw (4,1.4) node[scale=.7]{$\mathcal{C}_1$};

\draw (4,.75) node[scale=.7]{$\sigma$};
\draw (12.5,2.75) node[scale=.7]{$\sigma$};
\draw (12.5,2) node[scale=.7]{$\mathcal{C}_4$};

\draw (1,2) node{$\bullet$};
\draw (4,2) node{$\bullet$};
\draw (7,2) node{$\bullet$};

\draw (12,2.5) node{$\times$};
\draw (13,2.5) node{$\times$};

\draw (3.5,1) node{$\times$};
\draw (4.5,1) node{$\times$};

\end{tikzpicture}
\end{center}

Conversely, If a surface $q \in \mathcal{H}(3,1)$ has four cylinders, then all of them are stable and exactly one of them is absolute. 

\end{prop}

\begin{proof}
Let $q$ be a horizontally periodic surface in $\M$. Such a surface is given by Proposition \ref{existence} and up to applying imaginary rel, which preserves $\M$ by assumption, we can assume that the cylinders are stable. In particular, there are exactly two horizontal saddle connections that connect the order 1 singularity to itself and, up to rotating by an angle $\pi$, they appear on the bottom boundary component of two cylinders $\mathcal{C}'_2$ and $\mathcal{C}'_3$ of same type. The union of these two saddle connections also appears on the top component of a third relative cylinder of the opposite type that we denote by $\mathcal{C}'_1$. Notice that $c_1 = c_2 + c_3$, where $c_i$ denotes the circumference of $\mathcal{C}'_i$. 

\begin{center}
\begin{tikzpicture}[scale=.5]
\draw[dashed] (0,3) -- (3,3);
\draw (3,3) -- (4,2) -- (5,3);
\draw[dashed] (5,3) -- (8,3);
\draw (8,3) -- (7,2) -- (7,1);
\draw[dashed] (7,1) -- (1,1);
\draw (1,1) -- (1,2) -- (0,3);

\draw (2,2.4) node[scale=.7]{$\mathcal{C}'_2$};
\draw (6,2.4) node[scale=.7]{$\mathcal{C}'_3$};
\draw (4,1.4) node[scale=.7]{$\mathcal{C}'_1$};

\draw (1,2) node{$\bullet$};
\draw (4,2) node{$\bullet$};
\draw (7,2) node{$\bullet$};

\end{tikzpicture}
\end{center}

If these cylinders cover $X_q$, the top component of the cylinders $\mathcal{C}'_2$ and $\mathcal{C}'_3$ is glued to the bottom component of $\mathcal{C}'_1$. We claim that either $\mathcal{C}'_2$ or $\mathcal{C}'_3$ carries on its boundary a saddle connection that is smaller than its circumference. If that is not the case, it means that their top component is made of a single saddle connection and the total angle around the order 3 singularities would add up to $4\pi$, which is a contradiction. Let $s$ be such a saddle connection. Up to renaming the cylinders, we can assume that $s$ sits on the top component of $\mathcal{C}'_2$ and after shearing the surface $q$ we can assume that $s$ is the cross curve of a vertical and absolute cylinder that crosses $\mathcal{C}'_1$ and $\mathcal{C}'_2$ and that we denote by $\mathcal{C}_4$, as depicted in the following:

\begin{center}
\begin{tikzpicture}[scale=.5]

\begin{scope}[shift={(0,0)}]
\draw[dashed] (0,3) -- (3,3);
\draw (3,3) -- (4,2) -- (5,3);
\draw[dashed] (5,3) -- (8,3);
\draw (8,3) -- (7,2) -- (7,1);
\draw[dashed] (7,1) -- (1,1);
\draw (1,1) -- (1,2) -- (0,3);

\draw (1,3) node{$\times$}; 
\draw[thick] (1,3) -- (2,3);
\draw (2,3) node{$\times$};
\draw (1.5,3.4) node[scale=.7]{$s$};

\draw (3,1) node{$\times$};
\draw (4,1) node{$\times$};
\draw[thick] (3,1) -- (4,1);
\draw (3.5,0.6) node[scale=.7]{$s$};

\draw (1,2) node{$\bullet$};
\draw (4,2) node{$\bullet$};
\draw (7,2) node{$\bullet$};

\draw [-stealth](8.5,2) -- (9.5,2);
\draw (9,2.5) node{shear};

\end{scope}

\begin{scope}[shift={(10,0)}]

\draw[dashed] (1,3) -- (4,3);
\draw (4,3) -- (4,2) -- (5,3);
\draw[dashed] (5,3) -- (8,3);
\draw (8,3) -- (7,2) -- (7,1);
\draw[dashed] (7,1) -- (1,1);
\draw (1,1) -- (1,2) -- (1,3);

\draw[fill=black, opacity=.3] (2,1) rectangle (3,3);

\draw (2,3) node{$\times$}; 
\draw (3,3) node{$\times$};
\draw[thick] (2,3) -- (3,3);
\draw (2.5,3.4) node[scale=.7]{$s$};

\draw (2,1) node{$\times$};
\draw (3,1) node{$\times$};
\draw[thick] (2,1) -- (3,1);
\draw (2.5,0.6) node[scale=.7]{$s$};

\draw (1,2) node{$\bullet$};
\draw (4,2) node{$\bullet$};
\draw (7,2) node{$\bullet$};

\end{scope};

\end{tikzpicture}
\end{center}

By Proposition \ref{periodicity}, the vertical direction on $q$ is completely periodic and, up to applying real rel, we can assume the corresponding cylinder decomposition is stable. The same argument as in the beginning of the proof reveals that there are three other vertical relative cylinders having the order 1 singularity on one of their components, we denote them by $\mathcal{C}_1$, $\mathcal{C}_2$ and $\mathcal{C}_3$ with $\mathcal{C}_2$ and $\mathcal{C}_3$ having same type and the sum of the circumferences of $\mathcal{C}_2$ and $\mathcal{C}_3$ is equal to the one of $\mathcal{C}_1$. We can label the types so that $\mathcal{C}_1$ is in in $\mathfrak{C}^-$, $\mathcal{C}_2$ and $\mathcal{C}_3$ in $\mathfrak{C}^+$. Since there are at most four cylinders in a given direction for a surface in $\mathcal{H}(3,1)$, see for instance Lemma C.1 in \cite{aulicino2016rank1}, we deduce that these horizontal cylinders cover $X_q$.

Now, if $\mathcal{C}'_1$, $\mathcal{C}'_2$ and $\mathcal{C}'_3$ do not cover $X_q$, there has to be another horizontal cylinder $\mathcal{C}_4$ which has to be absolute. In any case we found a surface $q \in \M$ with four cylinders $\mathcal{C}_1, \cdots, \mathcal{C}_4$ as in the statement and such that $c_1 = c_2 + c_3$. The cylinder $\mathcal{C}_1$ being relative, a saddle connection on the bottom boundary component of $\mathcal{C}_1$ has to appear on the top component of either $\mathcal{C}_2$, $\mathcal{C}_3$ or $\mathcal{C}_4$ but because $c_1 = c_2 + c_3$ and the surface $q$ is connected, we deduce that there is a saddle connection $\sigma$ that is both on the bottom component of $\mathcal{C}_1$ and the top component of $\mathcal{C}_4$.

The last statement follows from our analysis of the cylinders that carry the order one singularity on their boundaries. It also follows from the list of 4-cylinder diagrams provided in Lemma 6.3 in \cite{aulicino2016rank}
\end{proof}

\begin{thm}[Theorem \ref{T2}]
Let $\M$ be a rel-invariant rank one affine invariant orbifold in $\mathcal{H}(3,1)$. Then $\M$ is arithmetic.
\end{thm}

\begin{proof}

Let $\M$ be a rel-invariant rank one affine invariant orbifold in $\mathcal{H}(3,1)$ and let $q$ be a surface in $\M$ as Proposition \ref{combinatorix31}. It is a consequence of Corollary \ref{notmixed} that the circumferences of the cylinders $\mathcal{C}_1$, $\mathcal{C}_2$ and $\mathcal{C}_3$ are pairwise commensurable. Shear the surface $q$ such that $\sigma$ sits below the singularity that is on the top boundary component in $\mathcal{C}_1$. Let $q'$ be the surface obtained by overcollapsing $\mathcal{C}_1$ using imaginary rel. We depict an example of this construction for the reader's convenience: 

\begin{center}
    \begin{tikzpicture}
    \begin{scope}[scale=.6]
    
    \draw (0,0) rectangle (3,1);
    \draw (0,2) -- (.5,3) -- (2.5,3) -- (2,2) -- cycle;
    \draw (0,4) rectangle (1,5);
    \draw (2,4) rectangle (3,5);
    
    \draw (1.5,.5) node[scale=.7]{$\mathcal{C}_4$};
    \draw (1.25,2.5) node[scale=.7]{$\mathcal{C}_1$};    
    \draw (.5,4.5) node[scale=.7]{$\mathcal{C}_2$};
    \draw (2.5,4.5) node[scale=.7]{$\mathcal{C}_3$};    
    
    \draw (0,0) node{$\times$};
    \draw (1,0) node{$\times$};
    \draw (2,0) node{$\times$};
    \draw (3,0) node{$\times$};
    \draw (0,1) node{$\times$};
    \draw (2,1) node{$\times$};
    \draw (3,1) node{$\times$};
    
    \draw (0,2) node{$\times$};
    \draw (2,2) node{$\times$};    
    \draw (0.5,3) node{$\bullet$};    
    \draw (1.5,3) node{$\bullet$};
    \draw (2.5,3) node{$\bullet$};    
    
    \draw (0,4) node{$\bullet$};
    \draw (1,4) node{$\bullet$};   
    \draw (0,5) node{$\times$};    
    \draw (1,5) node{$\times$};    
    
    \draw (2,4) node{$\bullet$};
    \draw (3,4) node{$\bullet$};
    \draw (2,5) node{$\times$};    
    \draw (3,5) node{$\times$};  
    
    \draw (.5,-.25) node[scale=.7]{$1$};
    \draw (1.5,-.25) node[scale=.7]{$2$};
    \draw (2.5,-.25) node[scale=.7]{$3$};
    \draw (1,1.25) node[scale=.7]{$\sigma$};
    \draw (2.5,1.25) node[scale=.7]{$3$};
    
    \draw (1,1.75) node[scale=.7]{$\sigma$};    
    \draw (1,3.25) node[scale=.7]{$5$};       
    \draw (2,3.25) node[scale=.7]{$6$};       
    
    \draw (.5,3.75) node[scale=.7]{$5$};
    \draw (.5,5.25) node[scale=.7]{$1$};
    \draw (2.5,3.75) node[scale=.7]{$6$};
    \draw (2.5,5.255) node[scale=.7]{$2$};   
    
    \draw[>=latex,->](4.5,2.75) -- (6,2.75) ;
\draw (5.25,3.4) node[scale=.6]{overcollapse $\mathcal{C}_1$};
    \draw (1.5,-1) node[scale=.7]{$\sigma$ sits below $\bullet$ in $\mathcal{C}_1$};
    
    \end{scope}
    
    \begin{scope}[shift={(4.5,0)}, scale=.6]
    
    \draw (0,0) -- (0,1) -- (.5,.75) -- (1.5,.75) -- (2,1) -- (3,1) -- (3,0) -- cycle;
    \draw (0,3) rectangle (1,5);
    \draw (2,3) -- (2,5) -- (3,5) -- (3,3) -- (2.5,3.25) -- cycle;
    
    \draw (2.5,.5) node[scale=.7]{$\mathcal{C}'_4$};
    \draw (.5,4) node[scale=.7]{$\mathcal{C}'_2$};
    \draw (2.5,4) node[scale=.7]{$\mathcal{C}'_3$};    
    
    \draw (0,0) node{$\times$};
    \draw (1,0) node{$\times$};
    \draw (2,0) node{$\times$};
    \draw (3,0) node{$\times$};
    \draw (0,1) node{$\times$};
    \draw (2,1) node{$\times$};
    \draw (3,1) node{$\times$};
    \draw (.5,.75) node{$\bullet$};
    \draw (1.5,.75) node{$\bullet$};    
    
    \draw (0,3) node{$\bullet$};
    \draw (1,3) node{$\bullet$};   
    \draw (0,5) node{$\times$};    
    \draw (1,5) node{$\times$};    
    
    \draw (2,3) node{$\bullet$};
    \draw (3,3) node{$\bullet$};
    \draw (2,5) node{$\times$};    
    \draw (3,5) node{$\times$};
    \draw (2.5,3.25) node{$\times$}; 
    
    \draw (.5,-.25) node[scale=.7]{$1$};
    \draw (1.5,-.25) node[scale=.7]{$2$};
    \draw (2.5,-.25) node[scale=.7]{$3$};
    \draw (2.5,1.25) node[scale=.7]{$3$};
    
    \draw (1,1) node[scale=.7]{$5$};       1
    
    \draw (.5,2.75) node[scale=.7]{$5$};
    \draw (.5,5.25) node[scale=.7]{$1$};
    \draw (2.5,5.255) node[scale=.7]{$2$};   
    
    \draw (1.5,-1) node[scale=.7]{$q'$};
    
    \end{scope}
    \end{tikzpicture}
\end{center}

The surface $q'$ has three cylinders $\mathcal{C}'_2$, $\mathcal{C}'_3$ and $\mathcal{C}_4'$ that come from $\mathcal{C}_2$, $\mathcal{C}_3$ and $\mathcal{C}_4$ plus an additional relative cylinder $\mathcal{C}$ of height $\varepsilon$ that was created by the overcollapse of $\mathcal{C}_1$. Notice that because we arranged that $\sigma$ sit below the top singularity of $\mathcal{C}_1$ and that no singularity can enter the bottom component of $\mathcal{C}_4$ during the overcollapse of $\mathcal{C}_1$ (a saddle connection on the bottom component of $\mathcal{C}_4$ is only attached to the top component of a cylinder in $\mathfrak{C}^+$ or to $\mathcal{C}_4$ itself), $\mathcal{C}'_4$ is also relative. Because we overcollapsed using rel and the circumference of a cylinder is an absolute period of $q$, we know that for any $i \in \{2,3,4\}$ the cylinder $\mathcal{C}'_i$ has circumference $c_i$. By Proposition \ref{combinatorix31}, the surface $q'$ has three relative cylinders and one absolute cylinder. Up to renaming the cylinders $\mathcal{C}_2$ and $\mathcal{C}_3$, we can assume that it is $\mathcal{C}_3'$ that is absolute. It follows from Proposition \ref{noneq} that the circumferences of $\mathcal{C}'_2$ and $\mathcal{C}'_4$ are commensurable \textit{i.e} that $c_2$ is commensurable to $c_4$. We already knew that $c_1$, $c_2$ and $c_3$ were pairwise commensurable so we deduce from Proposition \ref{field} applied to $q$ that $\M$ is arithmetic. 

\end{proof}  

\begin{cor}
If $q$ is a nonarithmetic Veech surface in $\mathcal{H}(3,1)$, then the subset $\GL \cdot \mathcal{F}_q$ is dense in $\mathcal{H}(3,1)$. 
\end{cor}

\begin{proof}
Let $\M$ be the closure of $\GL \cdot\mathcal{F}_{q}$. This is an affine invariant orbifold by \cite{eskin2015isolation} as it is connected, closed, and $\GL$ invariant, and it is rel-invariant.  The rank of $\M$ has to be at least 2 as otherwise Lemma 2.11 in \cite{lanneau2017finiteness} implies that $\M$ is nonarithmetic, in contradiction with Theorem \ref{T2}. However Theorem 1.1 in \cite{aulicino2016rank} states that there are no rank 2 affine invariant orbifolds in $\mathcal{H}(3,1)$ and then the rank of $\M$ is $3$. Finally, we invoke Theorem 1.1 in \cite{mirzakhani2018full} to deduce that $\M$ is the whole stratum $\mathcal{H}(3,1)$.  
\end{proof}

\section{Data Availability Statement}

We used the cylinder database provided by the sage package $\mathrm{surface\_dynamics}$. See \url{https://github.com/flatsurf/surface-dynamics} for more information.

\appendix

\section{Cylinder diagrams in $\mathcal{H}^{odd}(2,2)$}

We recall the list of all cylinder diagrams in the stratum $\mathcal{H}^{odd}(2,2)$. We used the database provided by the sage package $\mathrm{surface\_dynamics}$. See \url{https://github.com/flatsurf/surface-dynamics} for more information.

\begin{center}
\begin{tikzpicture}[scale=.7]
\begin{scope}[shift = {(0,0)}] 
\draw (0,0) rectangle (6,1);

\draw (0,0) node{$\bullet$};
\draw (1,0) node{$\bullet$}; 
\draw (2,0) node{$\times$};
\draw (3,0) node{$\times$};
\draw (4,0) node{$\times$};
\draw (5,0) node{$\bullet$};
\draw (6,0) node{$\bullet$};

\draw (0,1) node{$\bullet$};
\draw (1,1) node{$\bullet$};
\draw (2,1) node{$\times$};
\draw (3,1) node{$\times$};
\draw (4,1) node{$\times$};
\draw (5,1) node{$\bullet$};
\draw (6,1) node{$\bullet$};

\draw (.5,.25) node[scale = .7]{0}; 
\draw (.5,.75) node[scale = .7]{3};

\draw (1.5,.25) node[scale = .7]{5}; 
\draw (1.5,.75) node[scale = .7]{5};

\draw (2.5,.25) node[scale = .7]{2}; 
\draw (2.5,.75) node[scale = .7]{1};

\draw (3.5,.25) node[scale = .7]{1}; 
\draw (3.5,.75) node[scale = .7]{2};

\draw (4.5,.25) node[scale = .7]{4}; 
\draw (4.5,.75) node[scale = .7]{4};

\draw (5.5,.25) node[scale = .7]{3}; 
\draw (5.5,.75) node[scale = .7]{0};

\draw (3,-1) node{1.1};

\end{scope}%1.1

\begin{scope}[shift = {(8,0)}] 
\draw (0,0) rectangle (6,1);

\draw (0,0) node{$\bullet$};
\draw (1,0) node{$\times$}; 
\draw (2,0) node{$\bullet$};
\draw (3,0) node{$\times$};
\draw (4,0) node{$\bullet$};
\draw (5,0) node{$\times$};
\draw (6,0) node{$\bullet$};

\draw (0,1) node{$\times$};
\draw (1,1) node{$\bullet$};
\draw (2,1) node{$\times$};
\draw (3,1) node{$\bullet$};
\draw (4,1) node{$\times$};
\draw (5,1) node{$\bullet$};
\draw (6,1) node{$\times$};

\draw (.5,.25) node[scale = .7]{0}; 
\draw (.5,.75) node[scale = .7]{4};

\draw (1.5,.25) node[scale = .7]{5}; 
\draw (1.5,.75) node[scale = .7]{2};

\draw (2.5,.25) node[scale = .7]{2}; 
\draw (2.5,.75) node[scale = .7]{5};

\draw (3.5,.25) node[scale = .7]{3}; 
\draw (3.5,.75) node[scale = .7]{1};

\draw (4.5,.25) node[scale = .7]{1}; 
\draw (4.5,.75) node[scale = .7]{3};

\draw (5.5,.25) node[scale = .7]{4}; 
\draw (5.5,.75) node[scale = .7]{0};

\draw (3,-1) node{1.2};

\end{scope}%1.2

\begin{scope}[shift = {(16,0)}] 
\draw (0,0) rectangle (6,1);

\draw (0,0) node{$\bullet$};
\draw (1,0) node{$\times$}; 
\draw (2,0) node{$\bullet$};
\draw (3,0) node{$\times$};
\draw (4,0) node{$\times$};
\draw (5,0) node{$\bullet$};
\draw (6,0) node{$\bullet$};

\draw (0,1) node{$\times$};
\draw (1,1) node{$\times$};
\draw (2,1) node{$\bullet$};
\draw (3,1) node{$\bullet$};
\draw (4,1) node{$\times$};
\draw (5,1) node{$\bullet$};
\draw (6,1) node{$\times$};

\draw (.5,.25) node[scale = .7]{0}; 
\draw (.5,.75) node[scale = .7]{2};

\draw (1.5,.25) node[scale = .7]{5}; 
\draw (1.5,.75) node[scale = .7]{5};

\draw (2.5,.25) node[scale = .7]{3}; 
\draw (2.5,.75) node[scale = .7]{4};

\draw (3.5,.25) node[scale = .7]{2}; 
\draw (3.5,.75) node[scale = .7]{3};

\draw (4.5,.25) node[scale = .7]{1}; 
\draw (4.5,.75) node[scale = .7]{1};

\draw (5.5,.25) node[scale = .7]{4}; 
\draw (5.5,.75) node[scale = .7]{0};

\draw (3,-1) node{1.3};

\end{scope}%1.3

\begin{scope}[shift = {(0,-6)}] 
\draw (0,0) rectangle (5,1);
\draw (0,2) rectangle (2,3);

\draw (0,0) node{$\bullet$};
\draw (1,0) node{$\times$};
\draw (2,0) node{$\times$};
\draw (3,0) node{$\bullet$};
\draw (5,0) node{$\bullet$};

\draw (0,1) node{$\times$};
\draw (1,1) node{$\bullet$};
\draw (2,1) node{$\times$};
\draw (3,1) node{$\times$};
\draw (4,1) node{$\bullet$};
\draw (5,1) node{$\times$};

\draw (0,2) node{$\times$};
\draw (1,2) node{$\bullet$};
\draw (2,2) node{$\times$};
\draw (0,3) node{$\bullet$};
\draw (2,3) node{$\bullet$};

\draw (.5,.25) node[scale = .7]{0}; 
\draw (.5,.75) node[scale = .7]{2};

\draw (1.5,.25) node[scale = .7]{1}; 
\draw (1.5,.75) node[scale = .7]{5};

\draw (2.5,.25) node[scale = .7]{2}; 
\draw (2.5,.75) node[scale = .7]{1};

\draw (4,.25) node[scale = .7]{3}; 
\draw (3.5,.75) node[scale = .7]{4};

\draw (.5,2.25) node[scale = .7]{4}; 
\draw (1,2.75) node[scale = .7]{3};

\draw (1.5,2.25) node[scale = .7]{5}; 
\draw (4.5,.75) node[scale = .7]{0};

\draw (2.5,-1) node{2.1};
\end{scope}%2.1

\begin{scope}[shift = {(8,-6)}] 

\draw (0,0) rectangle (4,1);
\draw (0,2) rectangle (4,3);

\draw (0,0) node{$\bullet$};
\draw (1,0) node{$\times$};
\draw (2,0) node{$\bullet$};
\draw (3,0) node{$\times$};
\draw (4,0) node{$\bullet$};

\draw (0,1) node{$\times$};
\draw (3,1) node{$\bullet$};
\draw (4,1) node{$\times$};

\draw (0,2) node{$\bullet$};
\draw (1,2) node{$\times$};
\draw (4,2) node{$\bullet$};

\draw (0,3) node{$\times$};
\draw (1,3) node{$\bullet$};
\draw (2,3) node{$\times$};
\draw (3,3) node{$\bullet$};
\draw (4,3) node{$\times$};

\draw (.5,.25) node[scale = .7]{0};
\draw (1.5,.25) node[scale = .7]{1};
\draw (2.5,.25) node[scale = .7]{3};
\draw (3.5,.25) node[scale = .7]{4};

\draw (1.5,.75) node[scale = .7]{5};
\draw (3.5,.75) node[scale = .7]{3};

\draw (.5,2.25) node[scale = .7]{2};
\draw (2.5,2.25) node[scale = .7]{5};

\draw (.5,2.75) node[scale = .7]{4};
\draw (1.5,2.75) node[scale = .7]{2};
\draw (2.5,2.75) node[scale = .7]{1};
\draw (3.5,2.75) node[scale = .7]{0};

\draw (2,-1) node{2.2};
\end{scope}%2.2

\begin{scope}[shift={(16,-6)}]
\draw (0,0) rectangle (5,1);
\draw (0,2) rectangle (1,3);

\draw (0,0) node{$\bullet$};
\draw (1,0) node{$\times$};
\draw (2,0) node{$\bullet$};
\draw (3,0) node{$\times$};
\draw (4,0) node{$\times$};
\draw (5,0) node{$\bullet$};

\draw (0,1) node{$\times$};
\draw (1,1) node{$\bullet$};
\draw (2,1) node{$\times$};
\draw (3,1) node{$\bullet$};
\draw (4,1) node{$\bullet$};
\draw (5,1) node{$\times$};

\draw (0,2) node{$\bullet$};
\draw (1,2) node{$\bullet$};

\draw (0,3) node{$\times$};
\draw (1,3) node{$\times$};

\draw (.5,.25) node[scale = .7]{0};
\draw (1.5,.25) node[scale = .7]{2};
\draw (2.5,.25) node[scale = .7]{3};
\draw (3.5,.25) node[scale = .7]{4};
\draw (4.5,.25) node[scale = .7]{1};

\draw (.5,.75) node[scale = .7]{1};
\draw (1.5,.75) node[scale = .7]{3};
\draw (2.5,.75) node[scale = .7]{2};
\draw (3.5,.75) node[scale = .7]{5};
\draw (4.5,.75) node[scale = .7]{0};

\draw (.5,2.25) node[scale = .7]{5};
\draw (.5,2.75) node[scale = .7]{4};

\draw (2.5,-1) node{2.3};
\end{scope}%2.3

\begin{scope}[shift={(0,-12)}]

\draw (0,0) rectangle (5,1);
\draw (0,2) rectangle (3,3);

\draw (0,0) node{$\bullet$};
\draw (1,0) node{$\times$};
\draw (2,0) node{$\bullet$};
\draw (5,0) node{$\bullet$}; 

\draw (0,1) node{$\times$}; 
\draw (1,1) node{$\times$};
\draw (2,1) node{$\bullet$};
\draw (3,1) node{$\times$};
\draw (4,1) node{$\bullet$};
\draw (5,1) node{$\times$};

\draw (0,2) node{$\bullet$};
\draw (1,2) node{$\times$}; 
\draw (2,2) node{$\times$};
\draw (3,2) node{$\bullet$}; 

\draw (0,3) node{$\bullet$}; 
\draw (3,3) node{$\bullet$}; 

\draw (.5,.25) node[scale= .7]{0};
\draw (1.5,.25) node[scale= .7]{2};
\draw (3.5,.25) node[scale= .7]{4};

\draw (.5,.75) node[scale= .7]{5};
\draw (1.5,.75) node[scale= .7]{2};
\draw (2.5,.75) node[scale= .7]{1};
\draw (3.5,.75) node[scale= .7]{3};
\draw (4.5,.75) node[scale= .7]{0};

\draw (.5,2.25) node[scale= .7]{1};
\draw (1.5,2.25) node[scale= .7]{5};
\draw (2.5,2.25) node[scale= .7]{3};
\draw (1.5,2.75) node[scale= .7]{4};

\draw (2.5,-1) node{2.4};
\end{scope}%2.4

\begin{scope}[shift={(8,-12)}]

\draw (0,0) rectangle (4,1); 
\draw (0,2) rectangle (2,3); 

\draw (0,0) node{$\bullet$};
\draw (1,0) node{$\bullet$};
\draw (2,0) node{$\times$};
\draw (3,0) node{$\times$};
\draw (4,0) node{$\bullet$};

\draw (0,1) node{$\bullet$};
\draw (1,1) node{$\times$}; 
\draw (2,1) node{$\times$}; 
\draw (3,1) node{$\bullet$};
\draw (4,1) node{$\bullet$};

\draw (0,2) node{$\times$};
\draw (1,2) node{$\bullet$};
\draw (2,2) node{$\times$};

\draw (0,3) node{$\bullet$};
\draw (1,3) node{$\times$};
\draw (2,3) node{$\bullet$};

\draw (.5,.25) node[scale=.7]{0};
\draw (1.5,.25) node[scale=.7]{3};
\draw (2.5,.25) node[scale=.7]{1};
\draw (3.5,.25) node[scale=.7]{2};

\draw (.5,.75) node[scale=.7]{5};
\draw (1.5,.75) node[scale=.7]{1};
\draw (2.5,.75) node[scale=.7]{4};
\draw (3.5,.75) node[scale=.7]{0};

\draw (.5,2.25) node[scale=.7]{4};
\draw (1.5,2.25) node[scale=.7]{5};

\draw (.5,2.75) node[scale=.7]{3};
\draw (1.5,2.75) node[scale=.7]{2};

\draw (2,-1) node{2.5};
\end{scope}%2.5

\begin{scope}[shift={(16,-12)}]

\draw (0,0) rectangle (3,1);
\draw (0,2) rectangle (3,3); 

\draw (0,0) node{$\bullet$};
\draw (1,0) node{$\times$};
\draw (2,0) node{$\times$};
\draw (3,0) node{$\bullet$}; 

\draw (0,1) node{$\bullet$};
\draw (1,1) node{$\times$};
\draw (2,1) node{$\times$};
\draw (3,1) node{$\bullet$};

\draw (0,2) node{$\times$};
\draw (1,2) node{$\bullet$};
\draw (2,2) node{$\bullet$};
\draw (3,2) node{$\times$};

\draw (0,3) node{$\times$};
\draw (1,3) node{$\bullet$};
\draw (2,3) node{$\bullet$};
\draw (3,3) node{$\times$};

\draw (.5,.25) node[scale=.7]{0};
\draw (1.5,.25) node[scale=.7]{3};
\draw (2.5,.25) node[scale=.7]{4};

\draw (.5,.75) node[scale=.7]{5};
\draw (1.5,.75) node[scale=.7]{3};
\draw (2.5,.75) node[scale=.7]{1};

\draw (.5,2.25) node[scale=.7]{1};
\draw (1.5,2.25) node[scale=.7]{2};
\draw (2.5,2.25) node[scale=.7]{5};

\draw (.5,2.75) node[scale=.7]{4};
\draw (1.5,2.75) node[scale=.7]{2};
\draw (2.5,2.75) node[scale=.7]{0};

\draw (1.5,-1) node{2.6};
\end{scope}%2.6

\begin{scope}[shift={(0,-18)}]

\draw (0,0) rectangle (5,1);
\draw (0,2) rectangle (1,3);

\draw (0,0) node{$\bullet$};
\draw (1,0) node{$\bullet$};
\draw (2,0) node{$\times$};
\draw (3,0) node{$\times$};
\draw (4,0) node{$\bullet$};
\draw (5,0) node{$\bullet$};

\draw (0,1) node{$\bullet$};
\draw (1,1) node{$\bullet$};
\draw (2,1) node{$\times$};
\draw (3,1) node{$\times$};
\draw (4,1) node{$\bullet$};
\draw (5,1) node{$\bullet$};

\draw (0,2) node{$\times$};
\draw (1,2) node{$\times$};
\draw (0,3) node{$\times$};
\draw (1,3) node{$\times$};

\draw (.5,.25) node[scale=.7]{0};
\draw (1.5,.25) node[scale=.7]{3};
\draw (2.5,.25) node[scale=.7]{4};
\draw (3.5,.25) node[scale=.7]{2};
\draw (4.5,.25) node[scale=.7]{1};

\draw (.5,.75) node[scale=.7]{1};
\draw (1.5,.75) node[scale=.7]{3};
\draw (2.5,.75) node[scale=.7]{5};
\draw (3.5,.75) node[scale=.7]{2};
\draw (4.5,.75) node[scale=.7]{0};

\draw (.5,2.25) node[scale=.7]{5};
\draw (.5,2.75) node[scale=.7]{4};

\draw (2.5,-1) node{2.7};
\end{scope}%2.7

\begin{scope}[shift={(8,-18)}]

\draw (0,0) rectangle (4,1);
\draw (0,2) rectangle (5,3);

\draw (0,0) node{$\bullet$};
\draw (1,0) node{$\bullet$};
\draw (2,0) node{$\times$}; 
\draw (3,0) node{$\bullet$};
\draw (4,0) node{$\bullet$};

\draw (0,1) node{$\times$};
\draw (4,1) node{$\times$}; 

\draw (0,2) node{$\times$};
\draw (1,2) node{$\times$}; 
\draw (5,2) node{$\times$};

\draw (0,3) node{$\bullet$};
\draw (1,3) node{$\bullet$};
\draw (2,3) node{$\times$};
\draw (3,3) node{$\times$}; 
\draw (4,3) node{$\bullet$};
\draw (5,3) node{$\bullet$};

\draw (.5,.25) node[scale=.7]{0};
\draw (1.5,.25) node[scale=.7]{4};
\draw (2.5,.25) node[scale=.7]{5};
\draw (3.5,.25) node[scale=.7]{2};

\draw (2,.75) node[scale=.7]{3};

\draw (.5,2.25) node[scale=.7]{1};
\draw (3,2.25) node[scale=.7]{3};

\draw (.5,2.75) node[scale=.7]{2};
\draw (1.5,2.75) node[scale=.7]{4};
\draw (2.5,2.75) node[scale=.7]{1};
\draw (3.5,2.75) node[scale=.7]{5};
\draw (4.5,2.75) node[scale=.7]{0};

\draw (2,-1) node{2.8};
\end{scope}%2.8

\begin{scope}[shift={(16,-18)}]

\draw (0,0) rectangle (3,1);
\draw (0,2) rectangle (3,3);

\draw (0,0) node{$\bullet$};
\draw (1,0) node{$\bullet$};
\draw (2,0) node{$\bullet$};
\draw (3,0) node{$\bullet$};

\draw (0,1) node{$\times$};
\draw (1,1) node{$\times$};
\draw (2,1) node{$\times$};
\draw (3,1) node{$\times$};

\draw (0,2) node{$\times$};
\draw (1,2) node{$\times$};
\draw (2,2) node{$\times$};
\draw (3,2) node{$\times$};

\draw (0,3) node{$\bullet$};
\draw (1,3) node{$\bullet$};
\draw (2,3) node{$\bullet$};
\draw (3,3) node{$\bullet$};

\draw (.5,.25) node[scale=.7]{0};
\draw (1.5,.25) node[scale=.7]{5};
\draw (2.5,.25) node[scale=.7]{1};

\draw (.5,.75) node[scale=.7]{3};
\draw (1.5,.75) node[scale=.7]{4};
\draw (2.5,.75) node[scale=.7]{2};

\draw (.5,2.25) node[scale=.7]{2};
\draw (1.5,2.25) node[scale=.7]{4};
\draw (2.5,2.25) node[scale=.7]{3};

\draw (.5,2.75) node[scale=.7]{1};
\draw (1.5,2.75) node[scale=.7]{5};
\draw (2.5,2.75) node[scale=.7]{0};

\draw (1.5,-1) node{2.9};
 
\end{scope}%2.9

\begin{scope}[shift={(0,-26)}]

\draw (0,0) rectangle (4,1);
\draw (0,2) rectangle (3,3);
\draw (0,4) rectangle (1,5);

\draw (0,0) node{$\bullet$};
\draw (1,0) node{$\bullet$};
\draw (4,0) node{$\bullet$};

\draw (0,1) node{$\bullet$};
\draw (1,1) node{$\times$};
\draw (2,1) node{$\times$};
\draw (3,1) node{$\bullet$};
\draw (4,1) node{$\bullet$};

\draw (0,2) node{$\times$};
\draw (1,2) node{$\times$};
\draw (2,2) node{$\bullet$};
\draw (3,2) node{$\times$};

\draw (0,3) node{$\bullet$};
\draw (3,3) node{$\bullet$};

\draw (0,4) node{$\times$};
\draw (1,4) node{$\times$}; 

\draw (0,5) node{$\times$};
\draw (1,5) node{$\times$};

\draw (.5,.25) node[scale=.7]{0};
\draw (2.5,.25) node[scale=.7]{1};

\draw (.5,.75) node[scale=.7]{5};
\draw (1.5,.75) node[scale=.7]{4};
\draw (2.5,.75) node[scale=.7]{3};
\draw (3.5,.75) node[scale=.7]{0};

\draw (.5,2.25) node[scale=.7]{2};
\draw (1.5,2.25) node[scale=.7]{3};
\draw (2.5,2.25) node[scale=.7]{5};

\draw (1.5,2.75) node[scale=.7]{1};

\draw (.5,4.25) node[scale=.7]{4};

\draw (.5,4.75) node[scale=.7]{2};

\draw (2,-1) node{3.1};
\end{scope}%3.1

\begin{scope}[shift={(8,-26)}]

\draw (0,0) rectangle (4,1);
\draw (0,2) rectangle (2,3);
\draw (0,4) rectangle (3,5); 

\draw (0,0) node{$\bullet$};
\draw (1,0) node{$\bullet$};
\draw (4,0) node{$\bullet$};

\draw (0,1) node{$\bullet$};
\draw (1,1) node{$\times$};
\draw (2,1) node{$\times$}; 
\draw (3,1) node{$\bullet$};
\draw (4,1) node{$\bullet$};

\draw (0,2) node{$\times$};
\draw (1,2) node{$\bullet$};
\draw (2,2) node{$\times$};

\draw (0,3) node{$\times$};
\draw (2,3) node{$\times$}; 

\draw (0,4) node{$\times$};
\draw (2,4) node{$\times$};
\draw (3,4) node{$\times$}; 

\draw (0,5) node{$\bullet$};
\draw (3,5) node{$\bullet$};

\draw (.5,.25) node[scale=.7]{0};
\draw (2.5,.25) node[scale=.7]{1};

\draw (.5,.75) node[scale=.7]{5};
\draw (1.5,.75) node[scale=.7]{4};
\draw (2.5,.75) node[scale=.7]{2};
\draw (3.5,.75) node[scale=.7]{0};

\draw (.5,2.25) node[scale=.7]{2};
\draw (1.5,2.25) node[scale=.7]{5};

\draw (1,2.75) node[scale=.7]{3};

\draw (1,4.25) node[scale=.7]{3};
\draw (2.5,4.25) node[scale=.7]{4};

\draw (1.5,4.75) node[scale=.7]{1};

\draw (2,-1) node{3.2};
\end{scope}%3.2

\begin{scope}[shift={(16,-26)}]

\draw (0,0) rectangle (4,1); 
\draw (0,2) rectangle (2,3); 
\draw (0,4) rectangle (1,5);

\draw (0,0) node{$\times$};
\draw (1,0) node{$\bullet$};
\draw (2,0) node{$\times$};
\draw (4,0) node{$\times$};

\draw (0,1) node{$\bullet$};
\draw (1,1) node{$\bullet$};
\draw (2,1) node{$\times$};
\draw (3,1) node{$\times$};
\draw (4,1) node{$\bullet$};

\draw (0,2) node{$\bullet$};
\draw (1,2) node{$\bullet$};
\draw (2,2) node{$\bullet$};

\draw (0,3) node{$\times$};
\draw (2,3) node{$\times$};

\draw (0,4) node{$\times$};
\draw (1,4) node{$\times$};

\draw (0,5) node{$\bullet$};
\draw (1,5) node{$\bullet$};

\draw (.5,.25) node[scale=.7]{0};
\draw (1.5,.25) node[scale=.7]{1};
\draw (3,.25) node[scale=.7]{2};

\draw (.5,.75) node[scale=.7]{5};
\draw (1.5,.75) node[scale=.7]{1};
\draw (2.5,.75) node[scale=.7]{4};
\draw (3.5,.75) node[scale=.7]{0};

\draw (.5,2.25) node[scale=.7]{3};
\draw (1.5,2.25) node[scale=.7]{5};

\draw (1,2.75) node[scale=.7]{2};

\draw (.5,4.25) node[scale=.7]{4};
\draw (.5,4.75) node[scale=.7]{3};
\draw (2,-1) node{3.3};

\end{scope}%3.3

\end{tikzpicture}

\begin{tikzpicture}[scale=.7]

\begin{scope}[shift={(0,0)}]

\draw (0,0) rectangle (4,1); 
\draw (0,2) rectangle (1,3);
\draw (0,4) rectangle (1,5);

\draw (0,0) node{$\bullet$};
\draw (1,0) node{$\times$};
\draw (2,0) node{$\bullet$};
\draw (3,0) node{$\bullet$};
\draw (4,0) node{$\bullet$};

\draw (0,1) node{$\times$};
\draw (1,1) node{$\times$};
\draw (2,1) node{$\times$};
\draw (3,1) node{$\bullet$};
\draw (4,1) node{$\times$};

\draw (0,2) node{$\times$}; 
\draw (1,2) node{$\times$}; 

\draw (0,3) node{$\bullet$};
\draw (1,3) node{$\bullet$};

\draw (0,4) node{$\times$}; 
\draw (1,4) node{$\times$}; 

\draw (0,5) node{$\bullet$};
\draw (1,5) node{$\bullet$};

\draw (.5,.25) node[scale=.7]{0};
\draw (1.5,.25) node[scale=.7]{1};
\draw (2.5,.25) node[scale=.7]{3};
\draw (3.5,.25) node[scale=.7]{2};

\draw (.5,.75) node[scale=.7]{5};
\draw (1.5,.75) node[scale=.7]{4};
\draw (2.5,.75) node[scale=.7]{1};
\draw (3.5,.75) node[scale=.7]{0};

\draw (.5,2.25) node[scale=.7]{4};
\draw (.5,2.75) node[scale=.7]{2};
\draw (.5,4.25) node[scale=.7]{5};
\draw (.5,4.75) node[scale=.7]{3};

\draw (2,-1) node{3.4};
\end{scope}%3.4

\begin{scope}[shift={(8,0)}]

\draw (0,0) rectangle (3,1);
\draw (0,2) rectangle (4,3);
\draw (0,4) rectangle (3,5);

\draw (0,0) node{$\times$};
\draw (1,0) node{$\times$};
\draw (2,0) node{$\times$};
\draw (3,0) node{$\times$};

\draw (0,1) node{$\bullet$};
\draw (3,1) node{$\bullet$};

\draw (0,2) node{$\bullet$};
\draw (1,2) node{$\bullet$};
\draw (4,2) node{$\bullet$};

\draw (0,3) node{$\bullet$};
\draw (3,3) node{$\bullet$};
\draw (4,3) node{$\bullet$};

\draw (0,4) node{$\bullet$};
\draw (3,4) node{$\bullet$};

\draw (0,5) node{$\times$};
\draw (1,5) node{$\times$};
\draw (2,5) node{$\times$};
\draw (3,5) node{$\times$};

\draw (.5,.25) node[scale=.7]{0};
\draw (1.5,.25) node[scale=.7]{2};
\draw (2.5,.25) node[scale=.7]{1};

\draw (1.5,.75) node[scale=.7]{4};

\draw (.5,2.25) node[scale=.7]{3};
\draw (2.5,2.25) node[scale=.7]{4};

\draw (1.5,2.75) node[scale=.7]{5};
\draw (3.5,2.75) node[scale=.7]{3};

\draw (1.5,4.25) node[scale=.7]{5};

\draw (.5,4.75) node[scale=.7]{1};
\draw (1.5,4.75) node[scale=.7]{2};
\draw (2.5,4.75) node[scale=.7]{0};

\draw (1.5,-1) node{3.5};

\end{scope}%3.5

\begin{scope}[shift={(16,0)}]

\draw (0,0) rectangle (3,1);
\draw (0,2) rectangle (3,3);
\draw (0,4) rectangle (1,5);

\draw (0,0) node{$\bullet$};
\draw (1,0) node{$\bullet$};
\draw (2,0) node{$\bullet$};
\draw (3,0) node{$\bullet$};

\draw (0,1) node{$\times$}; 
\draw (2,1) node{$\times$}; 
\draw (3,1) node{$\times$}; 

\draw (0,2) node{$\times$}; 
\draw (1,2) node{$\times$}; 
\draw (3,2) node{$\times$}; 

\draw (0,3) node{$\bullet$};
\draw (1,3) node{$\bullet$};
\draw (2,3) node{$\bullet$};
\draw (3,3) node{$\bullet$};

\draw (0,4) node{$\times$};
\draw (1,4) node{$\times$};

\draw (0,5) node{$\times$};
\draw (1,5) node{$\times$};

\draw (.5,.25) node[scale=.7]{0};
\draw (1.5,.25) node[scale=.7]{2};
\draw (2.5,.25) node[scale=.7]{1};

\draw (1,.75) node[scale=.7]{5};
\draw (2.5,.75) node[scale=.7]{4};

\draw (.5,2.25) node[scale=.7]{3};
\draw (2,2.25) node[scale=.7]{5};

\draw (.5,2.75) node[scale=.7]{1};
\draw (1.5,2.75) node[scale=.7]{2};
\draw (2.5,2.75) node[scale=.7]{0};

\draw (.5,4.25) node[scale=.7]{4};
\draw (.5,4.75) node[scale=.7]{3};

\draw (1.5,-1) node{3.6};

\end{scope}%3.6

\begin{scope}[shift={(0,-10)}]

\draw (0,0) rectangle (3,1);
\draw (0,2) rectangle (3,3);
\draw (0,4) rectangle (2,5);

\draw (0,0) node{$\bullet$};
\draw (1,0) node{$\times$};
\draw (2,0) node{$\times$};
\draw (3,0) node{$\bullet$};

\draw (0,1) node{$\times$};
\draw (2,1) node{$\times$};
\draw (3,1) node{$\times$};

\draw (0,2) node{$\bullet$};
\draw (1,2) node{$\bullet$};
\draw (3,2) node{$\bullet$};

\draw (0,3) node{$\times$};
\draw (1,3) node{$\bullet$};
\draw (2,3) node{$\bullet$};
\draw (3,3) node{$\times$};

\draw (0,4) node{$\times$};
\draw (2,4) node{$\times$};

\draw (0,5) node{$\bullet$};
\draw (2,5) node{$\bullet$};

\draw (.5,.25) node[scale=.7]{0};
\draw (1.5,.25) node[scale=.7]{2};
\draw (2.5,.25) node[scale=.7]{3};

\draw (1,.75) node[scale=.7]{5};
\draw (2.5,.75) node[scale=.7]{2};

\draw (.5,2.25) node[scale=.7]{1};
\draw (2,2.25) node[scale=.7]{4};

\draw (.5,2.75) node[scale=.7]{3};
\draw (1.5,2.75) node[scale=.7]{1};
\draw (2.5,2.75) node[scale=.7]{0};

\draw (1,4.25) node[scale=.7]{5};
\draw (1,4.75) node[scale=.7]{4};

\draw (1.5,-1) node{3.7};

\end{scope}%3.7

\begin{scope}[shift={(8,-10)}]

\draw (0,0) rectangle (3,1);
\draw (0,2) rectangle (3,3);
\draw (0,4) rectangle (1,5);

\draw (0,0) node{$\bullet$};
\draw (1,0) node{$\times$};
\draw (3,0) node{$\bullet$};

\draw (0,1) node{$\times$};
\draw (1,1) node{$\times$};
\draw (2,1) node{$\bullet$};
\draw (3,1) node{$\times$};

\draw (0,2) node{$\bullet$};    
\draw (1,2) node{$\times$};
\draw (2,2) node{$\bullet$};
\draw (3,2) node{$\bullet$};

\draw (0,3) node{$\times$};
\draw (2,3) node{$\bullet$};
\draw (3,3) node{$\times$};

\draw (0,4) node{$\times$};
\draw (1,4) node{$\times$};

\draw (0,5) node{$\bullet$};
\draw (1,5) node{$\bullet$};

\draw (.5,.25) node[scale=.7]{0};
\draw (2,.25) node[scale=.7]{3};

\draw (.5,.75) node[scale=.7]{5};
\draw (1.5,.75) node[scale=.7]{4};
\draw (2.5,.75) node[scale=.7]{0};

\draw (.5,2.25) node[scale=.7]{1};
\draw (1.5,2.25) node[scale=.7]{4};
\draw (2.5,2.25) node[scale=.7]{2};

\draw (1,2.75) node[scale=.7]{3};
\draw (2.5,2.75) node[scale=.7]{1};

\draw (.5,4.25) node[scale=.7]{5};
\draw (.5,4.75) node[scale=.7]{2};

\draw (1.5,-1) node{3.8};

\end{scope}%3.8

\begin{scope}[shift={(16,-10)}]

\draw (0,0) rectangle (4,1);
\draw (0,2) rectangle (2,3);
\draw (0,4) rectangle (2,5);

\draw (0,0) node{$\bullet$};
\draw (1,0) node{$\times$};
\draw (2,0) node{$\bullet$};
\draw (4,0) node{$\bullet$};

\draw (0,1) node{$\times$};
\draw (2,1) node{$\times$};
\draw (3,1) node{$\bullet$};
\draw (4,1) node{$\times$};

\draw (0,2) node{$\bullet$};
\draw (1,2) node{$\times$};
\draw (2,2) node{$\bullet$};

\draw (0,3) node{$\bullet$};
\draw (2,3) node{$\bullet$};

\draw (0,4) node{$\times$};
\draw (2,4) node{$\times$};

\draw (0,5) node{$\times$}; 
\draw (1,5) node{$\bullet$};
\draw (2,5) node{$\times$}; 

\draw (.5,.25) node[scale=.7]{0};
\draw (1.5,.25) node[scale=.7]{3};
\draw (3,.25) node[scale=.7]{1};

\draw (1,.75) node[scale=.7]{5};
\draw (2.5,.75) node[scale=.7]{4};
\draw (3.5,.75) node[scale=.7]{0};

\draw (.5,2.25) node[scale=.7]{2};
\draw (1.5,2.25) node[scale=.7]{4};

\draw (1,2.75) node[scale=.7]{1};

\draw (1,4.25) node[scale=.7]{5};

\draw (.5,4.75) node[scale=.7]{3};
\draw (1.5,4.75) node[scale=.7]{2};
 
\draw (2,-1) node{3.9};

\end{scope}%3.9

\begin{scope}[shift={(0,-20)}]

\draw (0,0) rectangle (3,1);
\draw (0,2) rectangle (4,3);
\draw (0,4) rectangle (3,5);

\draw (0,0) node{$\bullet$};
\draw (1,0) node{$\times$};
\draw (2,0) node{$\times$}; 
\draw (3,0) node{$\bullet$};

\draw (0,1) node{$\bullet$};
\draw (3,1) node{$\bullet$};

\draw (0,2) node{$\bullet$};
\draw (1,2) node{$\bullet$};
\draw (4,2) node{$\bullet$};

\draw (0,3) node{$\times$};
\draw (3,3) node{$\times$};
\draw (4,3) node{$\times$};

\draw (0,4) node{$\times$};
\draw (3,4) node{$\times$};

\draw (0,5) node{$\times$};
\draw (1,5) node{$\bullet$};
\draw (2,5) node{$\bullet$};
\draw (3,5) node{$\times$};

\draw (.5,.25) node[scale=.7]{0};
\draw (1.5,.25) node[scale=.7]{3};
\draw (2.5,.25) node[scale=.7]{1};

\draw (1.5,.75) node[scale=.7]{5};

\draw (.5,2.25) node[scale=.7]{2};
\draw (2.5,2.25) node[scale=.7]{5};

\draw (1.5,2.75) node[scale=.7]{4};
\draw (3.5,2.75) node[scale=.7]{3};

\draw (1.5,4.25) node[scale=.7]{4};

\draw (.5,4.75) node[scale=.7]{1};
\draw (1.5,4.75) node[scale=.7]{2};
\draw (2.5,4.75) node[scale=.7]{0};

\draw (1.5,-1) node{3.10};

\end{scope}%3.10

\begin{scope}[shift={(8,-20)}]

\draw (0,0) rectangle (4,1);
\draw (0,2) rectangle (1,3);
\draw (0,4) rectangle (1,5);

\draw (0,0) node{$\bullet$};
\draw (1,0) node{$\times$};
\draw (2,0) node{$\times$};
\draw (3,0) node{$\bullet$};
\draw (4,0) node{$\bullet$};

\draw (0,1) node{$\times$};
\draw (1,1) node{$\times$};
\draw (2,1) node{$\bullet$};
\draw (3,1) node{$\bullet$};
\draw (1,1) node{$\times$};

\draw (0,2) node{$\bullet$};
\draw (1,2) node{$\bullet$};

\draw (0,3) node{$\bullet$};
\draw (1,3) node{$\bullet$};

\draw (0,4) node{$\times$};
\draw (1,4) node{$\times$};

\draw (0,5) node{$\times$};
\draw (1,5) node{$\times$};

\draw (.5,.25) node[scale=.7]{0};
\draw (1.5,.25) node[scale=.7]{3};
\draw (2.5,.25) node[scale=.7]{1};
\draw (3.5,.25) node[scale=.7]{2};

\draw (.5,.75) node[scale=.7]{5};
\draw (1.5,.75) node[scale=.7]{1};
\draw (2.5,.75) node[scale=.7]{4};
\draw (3.5,.75) node[scale=.7]{0};

\draw (.5,2.25) node[scale=.7]{4};
\draw (.5,2.75) node[scale=.7]{2};

\draw (.5,4.25) node[scale=.7]{5};
\draw (.5,4.75) node[scale=.7]{3};

\draw (2,-1) node{3.11};

\end{scope}%3.11

\begin{scope}[shift={(16,-20)}]

\draw (0,0) rectangle (2,1);
\draw (0,2) rectangle (2,3);
\draw (0,4) rectangle (2,5);

\draw (0,0) node{$\bullet$};
\draw (1,0) node{$\times$};
\draw (2,0) node{$\bullet$};

\draw (0,1) node{$\times$};
\draw (1,1) node{$\bullet$};
\draw (2,1) node{$\times$};

\draw (0,2) node{$\bullet$};
\draw (1,2) node{$\times$};
\draw (2,2) node{$\bullet$};

\draw (0,3) node{$\times$};
\draw (1,3) node{$\bullet$};
\draw (2,3) node{$\times$};

\draw (0,4) node{$\times$};
\draw (1,4) node{$\bullet$};
\draw (2,4) node{$\times$};

\draw (0,5) node{$\times$};
\draw (1,5) node{$\bullet$};
\draw (2,5) node{$\times$};

\draw (.5,.25) node[scale=.7]{0};
\draw (1.5,.25) node[scale=.7]{5};

\draw (.5,.75) node[scale=.7]{4};
\draw (1.5,.75) node[scale=.7]{3};

\draw (.5,2.25) node[scale=.7]{1};
\draw (1.5,2.25) node[scale=.7]{4};

\draw (.5,2.75) node[scale=.7]{2};
\draw (1.5,2.75) node[scale=.7]{0};

\draw (.5,4.25) node[scale=.7]{2};
\draw (1.5,4.25) node[scale=.7]{3};

\draw (.5,4.75) node[scale=.7]{5};
\draw (1.5,4.75) node[scale=.7]{1};

\draw (1,-1) node{3.12};

\end{scope}%3.12

\end{tikzpicture}

\begin{tikzpicture}[scale=.7]

\begin{scope}[shift={(0,0)}]

\draw (0,0) rectangle (2,1);
\draw (0,2) rectangle (2,3);
\draw (0,4) rectangle (2,5);

\draw (0,0) node{$\bullet$};
\draw (1,0) node{$\times$};
\draw (2,0) node{$\bullet$};

\draw (0,1) node{$\bullet$};
\draw (1,1) node{$\times$};
\draw (2,1) node{$\bullet$};

\draw (0,2) node{$\bullet$};
\draw (1,2) node{$\times$};
\draw (2,2) node{$\bullet$};

\draw (0,3) node{$\bullet$};
\draw (1,3) node{$\times$};
\draw (2,3) node{$\bullet$};

\draw (0,4) node{$\bullet$};
\draw (1,4) node{$\times$};
\draw (2,4) node{$\bullet$};

\draw (0,5) node{$\times$};
\draw (1,5) node{$\bullet$};
\draw (2,5) node{$\times$};

\draw (.5,.25) node[scale=.7]{0};
\draw (1.5,.25) node[scale=.7]{5};

\draw (.5,.75) node[scale=.7]{0};
\draw (1.5,.75) node[scale=.7]{4};

\draw (.5,2.25) node[scale=.7]{1};
\draw (1.5,2.25) node[scale=.7]{4};

\draw (.5,2.75) node[scale=.7]{1};
\draw (1.5,2.75) node[scale=.7]{3};

\draw (.5,4.25) node[scale=.7]{2};
\draw (1.5,4.25) node[scale=.7]{3};

\draw (.5,4.75) node[scale=.7]{5};
\draw (1.5,4.75) node[scale=.7]{2};

\draw (1,-1) node{3.13};

\end{scope}%3.13

\begin{scope}[shift={(8,0)}]

\draw (0,0) rectangle (3,1);
\draw (0,2) rectangle (1,3);
\draw (0,4) rectangle (2,5);
\draw (0,6) rectangle (2,7);

\draw (0,0) node{$\bullet$};
\draw (1,0) node{$\bullet$};
\draw (3,0) node{$\bullet$};

\draw (0,1) node{$\bullet$};
\draw (2,1) node{$\bullet$};
\draw (3,1) node{$\bullet$};

\draw (0,2) node{$\times$};
\draw (1,2) node{$\times$};

\draw (0,3) node{$\times$};
\draw (1,3) node{$\times$};

\draw (0,4) node{$\times$};
\draw (1,4) node{$\times$};
\draw (2,4) node{$\times$};

\draw (0,5) node{$\bullet$};
\draw (2,5) node{$\bullet$};

\draw (0,6) node{$\bullet$};
\draw (2,6) node{$\bullet$};

\draw (0,7) node{$\times$};
\draw (1,7) node{$\times$};
\draw (2,7) node{$\times$};

\draw (.5,.25) node[scale=.7]{0};
\draw (2,.25) node[scale=.7]{1};

\draw (1,.75) node[scale=.7]{5};
\draw (2.5,.75) node[scale=.7]{0};

\draw (.5,2.25) node[scale=.7]{2};
\draw (.5,2.75) node[scale=.7]{4};

\draw (.5,4.25) node[scale=.7]{3};
\draw (1.5,4.25) node[scale=.7]{4};

\draw (1,4.75) node[scale=.7]{1};

\draw (1,6.25) node[scale=.7]{5};

\draw (.5,6.75) node[scale=.7]{3};
\draw (1.5,6.75) node[scale=.7]{2};

\draw (2,-1) node{4.1};

\end{scope}%4.1

\begin{scope}[shift={(16,0)}]

\draw (0,0) rectangle (3,1);
\draw (0,2) rectangle (1,3);
\draw (0,4) rectangle (1,5);
\draw (0,6) rectangle (1,7);

\draw (0,0) node{$\bullet$}; 
\draw (1,0) node{$\bullet$}; 
\draw (2,0) node{$\bullet$}; 
\draw (3,0) node{$\bullet$}; 

\draw (0,1) node{$\times$}; 
\draw (1,1) node{$\times$}; 
\draw (2,1) node{$\times$}; 
\draw (3,1) node{$\times$}; 

\draw (0,2) node{$\times$};
\draw (1,2) node{$\times$};

\draw (0,3) node{$\bullet$};
\draw (1,3) node{$\bullet$};

\draw (0,4) node{$\times$};
\draw (1,4) node{$\times$};

\draw (0,5) node{$\bullet$};
\draw (1,5) node{$\bullet$};

\draw (0,6) node{$\times$};
\draw (1,6) node{$\times$};

\draw (0,7) node{$\bullet$};
\draw (1,7) node{$\bullet$};

\draw (.5,.25) node[scale=.7]{0};
\draw (1.5,.25) node[scale=.7]{2};
\draw (2.5,.25) node[scale=.7]{1};

\draw (.5,.75) node[scale=.7]{5};
\draw (1.5,.75) node[scale=.7]{4};
\draw (2.5,.75) node[scale=.7]{3};

\draw (.5,2.25) node[scale=.7]{3};
\draw (.5,2.75) node[scale=.7]{1};

\draw (.5,4.25) node[scale=.7]{4};
\draw (.5,4.75) node[scale=.7]{2};

\draw (.5,6.25) node[scale=.7]{5};
\draw (.5,6.75) node[scale=.7]{0};

\draw (1.5,-1) node{4.2};

\end{scope}%4.2

\begin{scope}[shift={(0,-10)}]

\draw (0,0) rectangle (2,1);
\draw (0,2) rectangle (1,3);
\draw (0,4) rectangle (3,5);
\draw (0,6) rectangle (2,7);

\draw (0,0) node{$\bullet$};
\draw (1,0) node{$\bullet$};
\draw (2,0) node{$\bullet$};

\draw (0,1) node{$\times$};
\draw (2,1) node{$\times$};

\draw (0,2) node{$\times$};
\draw (1,2) node{$\times$};

\draw (0,3) node{$\bullet$};
\draw (1,3) node{$\bullet$};

\draw (0,4) node{$\times$};
\draw (1,4) node{$\times$};
\draw (3,4) node{$\times$};

\draw (0,5) node{$\bullet$};
\draw (2,5) node{$\bullet$};
\draw (3,5) node{$\bullet$};

\draw (0,6) node{$\bullet$};
\draw (2,6) node{$\bullet$};

\draw (0,7) node{$\times$};
\draw (1,7) node{$\times$};
\draw (2,7) node{$\times$};

\draw (.5,.25) node[scale=.7]{0};
\draw (1.5,.25) node[scale=.7]{3};

\draw (1,.75) node[scale=.7]{5};

\draw (.5,2.25) node[scale=.7]{1};
\draw (.5,2.75) node[scale=.7]{0};

\draw (.5,4.25) node[scale=.7]{2};
\draw (2,4.25) node[scale=.7]{5};

\draw (1,4.75) node[scale=.7]{4};
\draw (2.5,4.75) node[scale=.7]{3};

\draw (1,6.25) node[scale=.7]{4};

\draw (.5,6.75) node[scale=.7]{2};
\draw (1.5,6.75) node[scale=.7]{1};
 
\draw (1,-1) node{4.3};

\end{scope}%4.3

\begin{scope}[shift={(8,-10)}]

\draw (0,0) rectangle (2,1);
\draw (0,2) rectangle (2,3);
\draw (0,4) rectangle (1,5);
\draw (0,6) rectangle (1,7);

\draw (0,0) node{$\bullet$};
\draw (1,0) node{$\bullet$};
\draw (2,0) node{$\bullet$};

\draw (0,1) node{$\bullet$};
\draw (1,1) node{$\bullet$};
\draw (2,1) node{$\bullet$};

\draw (0,2) node{$\times$};
\draw (1,2) node{$\times$};
\draw (2,2) node{$\times$};

\draw (0,3) node{$\times$};
\draw (1,3) node{$\times$};
\draw (2,3) node{$\times$};

\draw (0,4) node{$\times$};
\draw (1,4) node{$\times$};

\draw (0,5) node{$\bullet$};
\draw (1,5) node{$\bullet$};

\draw (0,6) node{$\bullet$};
\draw (1,6) node{$\bullet$};

\draw (0,7) node{$\times$};
\draw (1,7) node{$\times$};

\draw (.5,.25) node[scale=.7]{0};
\draw (1.5,.25) node[scale=.7]{3};

\draw (.5,.75) node[scale=.7]{5};
\draw (1.5,.75) node[scale=.7]{0};

\draw (.5,2.25) node[scale=.7]{1};
\draw (1.5,2.25) node[scale=.7]{2};

\draw (.5,2.75) node[scale=.7]{1};
\draw (1.5,2.75) node[scale=.7]{4};

\draw (.5,4.25) node[scale=.7]{4};
\draw (.5,4.75) node[scale=.7]{3};

\draw (.5,6.25) node[scale=.7]{5};
\draw (.5,6.75) node[scale=.7]{2};

\draw (1,-1) node{4.4};

\end{scope}%4.4

\begin{scope}[shift={(16,-10)}]

\draw (0,0) rectangle (2,1);
\draw (0,2) rectangle (2,3);
\draw (0,4) rectangle (1,5);
\draw (0,6) rectangle (1,7);

\draw (0,0) node{$\bullet$};
\draw (1,0) node{$\bullet$};
\draw (2,0) node{$\bullet$};

\draw (0,1) node{$\times$};
\draw (1,1) node{$\times$};
\draw (2,1) node{$\times$};

\draw (0,2) node{$\times$};
\draw (1,2) node{$\times$};
\draw (2,2) node{$\times$};

\draw (0,3) node{$\bullet$};
\draw (1,3) node{$\bullet$};
\draw (2,3) node{$\bullet$};

\draw (0,4) node{$\bullet$};
\draw (1,4) node{$\bullet$};

\draw (0,5) node{$\bullet$};
\draw (1,5) node{$\bullet$};

\draw (0,6) node{$\times$};
\draw (1,6) node{$\times$};

\draw (0,7) node{$\times$};
\draw (1,7) node{$\times$};

\draw (.5,.25) node[scale=.7]{0};
\draw (1.5,.25) node[scale=.7]{5};

\draw (.5,.75) node[scale=.7]{4};
\draw (1.5,.75) node[scale=.7]{3};

\draw (.5,2.25) node[scale=.7]{1};
\draw (1.5,2.25) node[scale=.7]{4};

\draw (.5,2.75) node[scale=.7]{5};
\draw (1.5,2.75) node[scale=.7]{2};

\draw (.5,4.25) node[scale=.7]{2};
\draw (.5,4.75) node[scale=.7]{0};

\draw (.5,6.25) node[scale=.7]{3};
\draw (.5,6.75) node[scale=.7]{1};

\draw (1,-1) node{4.5};

\end{scope}%4.5

\end{tikzpicture}
\end{center} %list of cylinder diagrams in H(2,2)^odd

\end{document}